\documentclass[11pt,letterpaper]{amsart}
\usepackage{amssymb}
\usepackage{latexsym}
\usepackage{tikz}
\usepackage{anysize}
\usepackage[T1]{fontenc}
\usepackage[sc]{mathpazo}
\usepackage{color}
\usepackage[colorlinks]{hyperref}
\usepackage{amsfonts}
\usepackage{graphicx}
\usepackage{float}
\usepackage{subfig}
\usepackage{comment}
\definecolor {refcol}{RGB}{40,0,255}
\hypersetup{colorlinks=true,allcolors=refcol}
\makeatletter
\newfont{\footsc}{cmcsc10 at 8truept}
\newfont{\footbf}{cmbx10 at 8truept}
\newfont{\footrm}{cmr10 at 10truept}
\newtheorem{theorem}{Theorem}

\newtheorem{corollary}{Corollary}

\newtheorem{definition}{Definition}
\newtheorem{example}{Example}

\newtheorem{lemma}{Lemma}

\newtheorem{proposition}[theorem]{Proposition}

\newtheorem{remark}{Remark}

\begin{document}
%\nocite{*} check non-used citations

\title{Distribution of deformed Laplacian  limit points}

\author{Elismar R. Oliveira$^1$,~~\ Jonas Szutkoski $^2$, ~~\ Vilmar Trevisan$^1$\\
\small{$^{1}${\it Instituto de Matem\'atica e Estat\'{\i}stica, Universidade Federal do Rio Grande do Sul,}} \\
\small{Porto Alegre, Brasil} \\
\small{\texttt{elismar.oliveira@ufrgs.br,~~trevisan@mat.ufrgs.br}}\\
\small{$^{2}${\it DECESA-UFCSPA,}} \\
\small{\noindent  \texttt{szutkoski@ufcspa.edu.br}}
}
\date{}
\maketitle
\begin{abstract}
This paper investigates  limit points of the deformed Laplacian matrix, which merges the Laplacian and signless Laplacian matrices of a graph through a quadractic one-parameter family of matrices.   First, we show that any value greater or equal to 1 is a deformed Laplacian limit point (for different values of the parameter $s$)  using a simple family of trees. Second, we define  $(T_k)_{k \in \mathbb{N}}$ the Shearer's sequence of caterpillars for $\lambda>1$ and we present a convergence criterion based on Shearer's approach. Our main result is that for any fixed value $\lambda_0>1$ there exists a unique value $0<s^* <\sqrt{\lambda_0} -1$ such that,  and for any $s \in (0,s^*)$ the interval $[\lambda_0, \; +\infty)$ is entirely formed by $s$-deformed Laplacian limit points (for the same value of $s$). Finally, we provide some numerical data exploring the limit properties.
\end{abstract}


\section{Introduction}

For a given graph $G$ with $n$ vertices, adjacency matrix $A$ and  diagonal degree matrix $D$, the {\it deformed Laplacian matrix} of $G$ is defined as
\begin{equation}\label{eq:MG}
  M_G(s) =  I - sA + s^2(D - I),
\end{equation}
where $I$ is the $n \times n$ identity matrix and $s \in \mathbb{R}$.  $M_G(s) $ was defined and studied in a distinct context (see, for example, \cite{Mor2013,GDV2018}). Our point of view is to see and understand this parametrized family  of matrices as an interpolation between classical graph matrices.  As $M_G(-1) = A +D$ is the signless Laplacian matrix and $M_G(1) = D-A$ is the combinatorial Laplacian matrix, the deformed Laplacian matrix may provide a smooth transition between these two well studied graph matrices.

We recall that the ground breaking work of V. Nikiforov~\cite{nikiforov2017Aalpha}, who introduced the $A_\alpha$ matrix of a graph, defined as $A_G(\alpha) := \alpha D(G) + (1 - \alpha) A(G)$, where $0 \leq \alpha \leq 1$ has been source of study for understanding the transition between the adjacency matrix $A(G)$ and the signless Laplacian $Q(G) = D(G) + A(G)$, and has opened new directions in spectral graph theory.
Our main motivation is that the deformed Laplacian matrix offers a promising avenue for generalizing and unifying spectral results associated for both matrices. Moreover, the structure of the deformed Laplacian reveals new algebraic features that may have applications not only in theoretical studies, but also in practical areas such as network robustness, centrality measures, and graph-based machine learning.

This paper may be seen as a sequel of the publication \cite{DOT2025} co-authored by two of the present authors, in which we explore spectral properties of this matrix. Our main goal here is to study the distribution of {\it limit points of the deformed Laplacian matrix}, which is part of a project called {\it Hoffman Program} \cite{wang2020hoffman}. In order to explain our results, we first discuss briefly the context.

In 1972 \cite{hoffman1972limit}, Hoffman asked which real numbers are limit points for the spectral radius of matrices having non negative integral entries. As he proved that it is only necessary to consider adjacency matrices of graphs, we say that Hoffman introduced the concept of limit points of eigenvalues of graphs. Hoffman’s pioneering results gave rise to a substantial body of literature, extending the study of limit points to other graph-related matrices and to eigenvalues beyond the spectral radius. For an overview, we refer to the survey \cite{wang2020hoffman} , as well as the more recent works \cite{WangLiuBel2020}, \cite{lineartrees-24} and \cite{Bar25}.

In order to generalize the concept of limit points of graph eigenvalues, let $M$ be a matrix associated to a graph $G$ such that the $M$-eigenvalues are all real and the $M$-spectral radius is equal to the largest $M$-eigenvalue, denoted by $\rho_M(G)$. A real number $\lambda(M)$ is an $M$-limit point of the $M$-spectral radius of graphs if there exists a sequence of graphs $(G_k)_{k \in \mathbb{N}}$ such that
$$\rho_M(G_i) \not= \rho_M(G_j) \mbox{ whenever } i \not= j \mbox{ and }
\lim_{k \to \infty }\rho_M(G_k) =\lambda(M).$$

The Hoffman program consists of two subproblems: i) determine all the possible values of $M$-limit point (if there are any) and; ii) Find all the connected graphs whose $M$-spectral radius does not exceed a fixed $M$-limit point.

In this paper, we discuss subproblem i) for the deformed Laplacian matrix. More precisely, we aim to characterize which real numbers $\lambda$ are $M$-limit points, for $M$ the deformed Laplacian matrix.

\begin{definition}\label{def:limit point}
	A number $\lambda \in \mathbb{R}$ is an $s$-deformed Laplacian  limit point if, for a fixed value $s \in \mathbb{R}$, there exists a sequence of graphs $(G_n)_{n \in \mathbb{N}}$ such that $\rho(M_{G_i}(s)) \neq \rho(M_{G_j}(s)),\; i\neq j$ and $\displaystyle\lim_{n \to \infty} \rho(M_{G_n}(s)) = \lambda$. The set of all $s$-deformed Laplacian  limit points is denoted by $\mathcal{L}(s)$.
\end{definition}

In this paper, we study and characterize which $\lambda >1$ are $s$-deformed limit points. Since the matrix $M_G(s)$ depends on the real parameter value of $s$, our results have two distinct directions. We first show that there is a sequence of graphs  such that any $\lambda >1$ is an $s$-deformed Laplacian  limit point for the same sequence, but for different values of $s$. For our second type of results, we fix values of $\lambda>1$  and, by constructing a sequence of graphs,  we find intervals of $s$-deformed Laplacian limit points for some values of $s \in \mathbb{R}$. To be more precise, we now describe our main results.

\subsection{Main results}

We introduce  one of the simplest types of graphs, other than a path, having a single vertex of degree 3.
\begin{definition}\label{def:T1nn}
	We define, for each $n \geq 1$, the family of trees $T_{1,n,n}$ formed by a vertex $u$ where we attach a path $P_1$ and two paths $P_n$ (see Figure~\ref{fig:Starlike_t1nn}).
\end{definition}
\begin{figure}[H]
	\centering
	\includegraphics[scale=1.2]{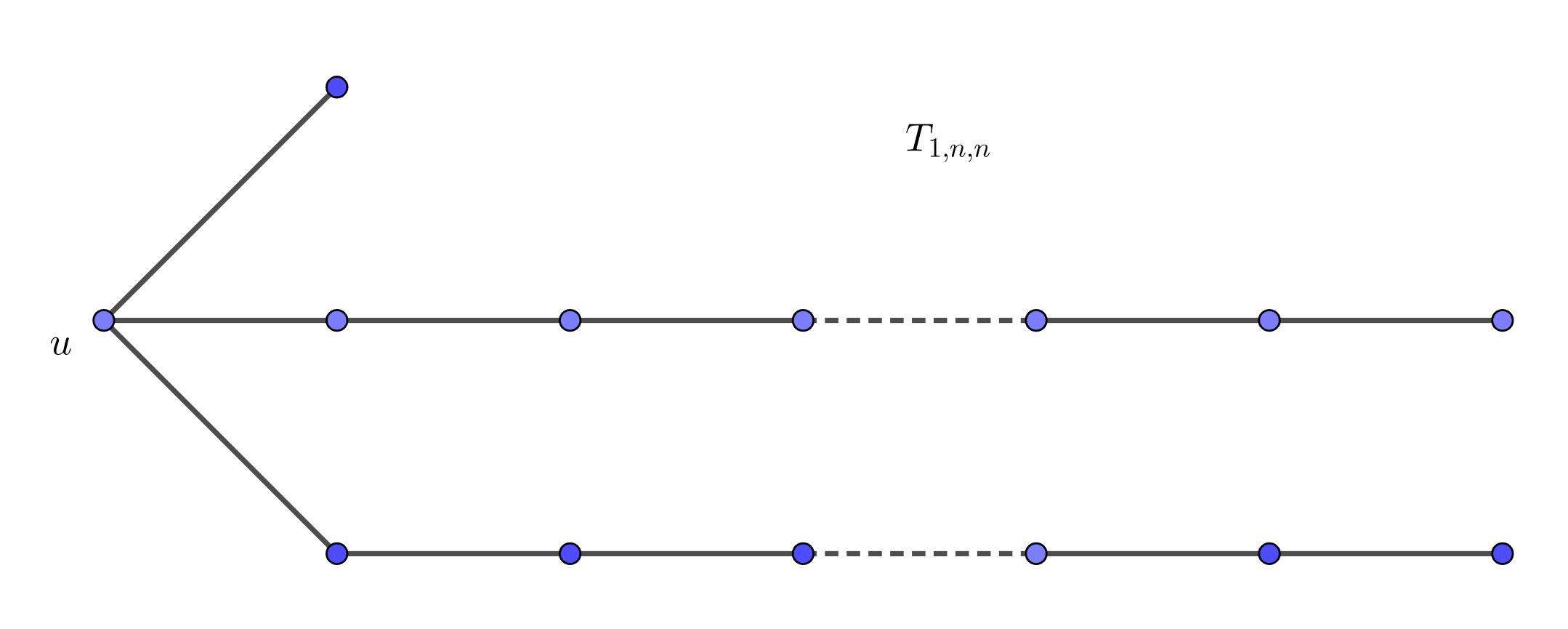}
	\caption{Starlike tree $T_{1,n,n}$.}
	\label{fig:Starlike_t1nn}
\end{figure}

In Section~\ref{sec: simple family} we prove Theorem~\ref{thm:all1} showing that any $\lambda >1$ is an $s$-deformed Laplacian limit point.

\begin{theorem}\label{thm:all1} Let $1 < \lambda \in \mathbb{R}$ be a real number. There exists an $s>0$ such that $\lambda$ is an $s$-deformed limit point. Precisely, $\lambda = \lim_{n \to \infty} \rho(M_{T_{1,n,n}}(s))$, for some $s\in \mathbb{R}$.
\end{theorem}

A much harder problem is, for fixed $\lambda >1$ and $s \in \mathbb{R}$, to construct a sequence of graphs $(G_k)_{k \in \mathbb{N}}$ such that $\rho(M_{G_k}(s)) \to \lambda$. By adapting Shearer's sequence of caterpillars (see Section~\ref{sec: Deformed Laplacian limit points for caterpillars} for definitions) we achieve this. In fact we prove a stronger result. For a fixed $\lambda >1$, we find a unique $s^*(\lambda)<1$ such that for any $s \in (0,s^*)$, the whole interval $[\lambda,+\infty)$ is composed by $s$-deformed limit points. The precise result is stated in Theorem~\ref{thm:suf criterion} and is proven in Section~\ref{sec: Approximation}.

\bigskip

\noindent The structure of the paper is organized as follows.

In Section~\ref{sec:preliminaries}, we review the key definitions and notational conventions. We focus on the spectral radius of deformed Laplacian matrix of graphs, we present the main tool we use in this manuscript and derive preliminary results, preparing for the main achievements in the subsequent sections.

In Section~\ref{sec: simple family} we use the previous results and prove Theorem~\ref{thm:all1}. In section~\ref{sec: Deformed Laplacian limit points for caterpillars} we present Shearer's approach, by defining a sequence of caterpillars for a given $\lambda>1$, and we establish our strategy of approximation through Shearer's sequences  for  $0<s<\sqrt{\lambda}-1$. In Section~\ref{sec: Approximation} we present our main results.  Theorem~\ref{thm: main criterion}  gives a criterion for convergence while Theorem~\ref{thm:suf criterion}  establishes the conditions and intervals formed by $s$-deformed Laplacian limit points. Finally, in Section~\ref{sec: Remarks and numerical data} we present some remarks, open questions and  numerical computations.

\section{Preliminaries}\label{sec:preliminaries}
For a graph $G$ and a real scalar $s$, we consider the deformed Laplacian matrix $M_G(s)$, given in Equation~\ref{eq:MG}
\begin{align}\label{MG2}
M_G(s)&= \left[
    \begin{array}{ccccc}
      1+ s^2 (d_{v_1}-1) & -s a_{1 \, 2} & \cdots & -s a_{1\, n-1} & -s a_{1 \, n} \\
      -s a_{1 \, 2} & 1 + s^2 (d_{v_2}-1) & \cdots & -s a_{2\, n-1} & -s a_{2 \,n} \\
      \vdots & \vdots & \ddots & \vdots & \vdots \\
      -s a_{n-1\,1} & -s a_{n-1\,2} & \cdots & 1+ s^2 (d_{v_{n-1}}-1) & -s a_{n-1\,n} \\
      -s a_{n \,1} & -s a_{n\,2} & \cdots & -s a_{n\,n-1} & 1+ s^2 (d_{v_n}-1) \\
    \end{array}
  \right].
\end{align}

Since $M_G(s)$  is real and symmetric, its eigenvalues are real. The multiset of eigenvalues of $M_G(s)$ is called the deformed Laplacian spectrum of $G$, denoted by $Spect(M_G(s)):=\sigma(M_G(s))$.

As this paper is devoted to study which numbers are limit points for the spectral radius, we first discuss Perron-Frobenius Theory for $M_G(s)$.

\subsection{Perron-Frobenius for the deformed Laplacian matrix}\label{sec:PF}

The classical theory of Perron-Frobenius is applied for square nonnegative irreducible matrices. So, as long as the graph $G$ is connected and $s\leq 0$, Perron-Frobenius theory is applied to $M_G(s)$.

The graphs we consider in this paper are trees, therefore connected, so irreducibility follows. In \cite{DOT2025} it is proven that $\sigma(M_G(s))=\sigma(M_G(-s))$ whenever $G$ is bipartite. Since trees are bipartite, we may consider only $s\leq 0$, implying nonnegativity. In particular, for trees, the largest eigenvalue of $M_G(s)$ is the spectral radius of $M_G(s)$, is a simple eigenvalue and is denoted by $\rho(M_G(s))$.

Therefore, the subproblem i) of the Hoffman Program is well defined for the $s$-deformed Laplacian matrix when we deal with trees.

It is easy to see that  $M_G(0)=I$, $M_G(1)=I - A + (D-I)=D-A=L(G)$ is the combinatorial Laplacian matrix of $G$ and $M_G(-1)=I + A + (D-I)=D+A=Q(G)$ is the signless Laplacian matrix of $G$.  Clearly the intervals (-1,0), (0,1) and their complements play a central role with respect to spectral properties of $M(G)$. This motivates the following.

\begin{definition}\label{def:sub} For a given graph $G$, consider $M_G(s)$.
  We say that $s\in \mathbb{R}$ is sub-Laplacian if $|s|<1$. Analogously, $s$ is super-Laplacian if $|s|>1$.
\end{definition}

We also notice that the symmetric matrix $M_G(s)$ underlies the graph $G$, it is a weighted matrix of $G$.  In the sequel, we explain the main tool to obtain our results. It is an algorithm to locate eigenvalues of symmetric matrices, and has been revealed a powerful technique to explore and understand spectral properties of the deformed Laplacian matrix.

\subsection{Main tool}\label{sec:maintool}

Our main tool in this paper is the Diagonalize algorithm to locate eigenvalues from \cite{JT2011} to study  eigenvalues of symmetric matrices $G$ . The Diagonalize algorithm (see Figure~\ref{fig:diagonalize}) is given below.
\begin{figure}[H]
{ \tt 	{\footnotesize
\begin{tabbing}
aaa\=aaa\=aaa\=aaa\=aaa\=aaa\=aaa\=aaa\= \kill
     \> Input: matrix $M = (m_{ij})$ with underlying tree $T$\\
     \> Input: Bottom up ordering $v_1,\ldots,v_n$ of $V(T)$\\
     \> Input: real number $x$ \\
     \> Output: diagonal matrix $D = \mbox{diag}(d_1, \ldots, d_n)$ congruent to $M + xI$ \\
     \> \\
     \>  Algorithm $\mbox{Diagonalize}(M, x)$ \\
     \> \> initialize $d_i := m_{ii} + x$, for all $i$ \\
     \> \> {\bf for } $k = 1$ to $n$ \\
     \> \> \> {\bf if} $v_k$ is a leaf {\bf then} continue \\
     \> \> \> {\bf else if} $d_c \neq 0$ for all children $c$ of $v_k$ {\bf then} \\
     \> \> \>  \>   $d_k := d_k - \sum \frac{(m_{ck})^2}{d_c}$, summing over all children of $v_k$ \\
     \> \> \> {\bf else } \\
     \> \> \> \> select one child $v_j$ of $v_k$ for which $d_j = 0$  \\
     \> \> \> \> $d_k  := -\frac{(m_{jk})^2}{2}$ \\
     \> \> \> \> $d_j  :=  2$ \\
     \> \> \> \> if $v_k$ has a parent $v_\ell$, remove the edge $\{v_k,v_\ell\}$. \\
     \> \>  {\bf end loop} \\
\end{tabbing}}
}
\caption{Algorithm $\mbox{Diagonalize}(M, x)$.}\label{fig:diagonalize}
\end{figure}
The following result summarizes the way in which the algorithm will be applied. Its proof is based on a property of matrix congruence known as Sylvester's Law of Inertia, we refer to~\cite{JT2011},~\cite{BragaRodrigues2017} and ~\cite{HopJacTrevBook22} for details.
\begin{theorem}\label{thm:inertia}
Let $M$ be a symmetric matrix of order $n$ that corresponds to a weighted tree $T$ and let $x$ be a real number. Given a bottom-up ordering of $T$, let $D$ be the diagonal matrix produced by Algorithm Diagonalize with entries $T$ and $x$. Then the number of positive/negative/zero entries of $D$ is number of eigenvalues of $M$ greater than/smaller than/equal to $-x$.

\end{theorem}

\begin{example}\label{ex:T1nn}
Consider the graph $T_{1,n,n}$ given in Figure \ref{fig:Starlike_t1nn}, choose $\lambda = \rho(M_{T_{1,n.n}}(s))$  the spectral radius of $T_{1,n,n}$ and apply the  Algorithm Diagonalize ($\mbox{Diag}(M_{T_{1,n,n}}(s), -\lambda)$).  In Figure~\ref{fig:diag_t1nn} we illustrate the application. On the left side, we have the initial values of vertices and edges. On the right side, we have the sequence of the outputs produced by the application of the algorithm, which we analyse now.
\begin{figure}[H]
	\centering
	\includegraphics[width=0.5\linewidth]{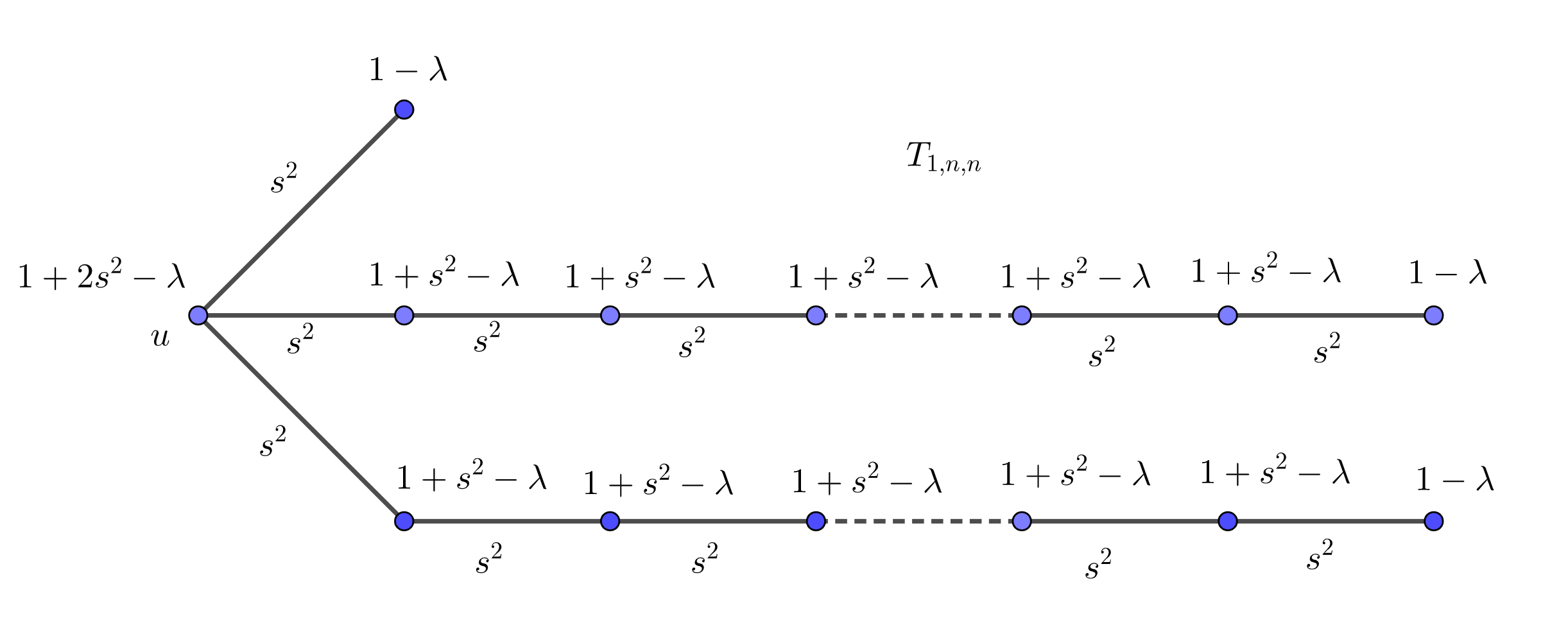}
	\includegraphics[width=0.45\linewidth]{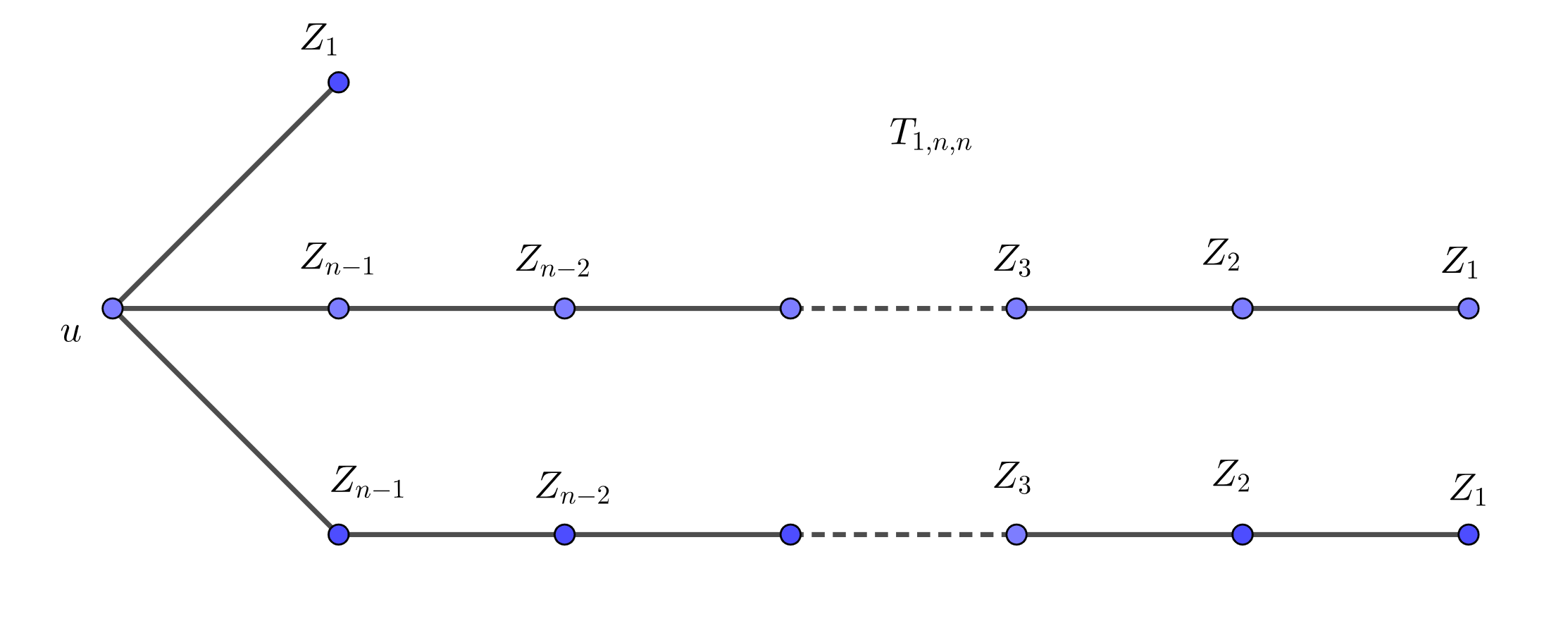}
	\caption{Initialization (left) and output (right) of $\mbox{Diag}(M_{T_{1,n,n}}(s), -\lambda_n)$ for $T_{1,n,n}$.}
	\label{fig:diag_t1nn}
\end{figure}
We see that $Z_1=1-\lambda, ~Z_2=1+s^2 -\lambda-\frac{s^2}{1-\lambda}$ and in general $Z_j=1+s^2 -\lambda-\frac{s^2}{Z_{j-1}}$, for $j=2,\ldots, n-1$. It is key to notice that because $\lambda$ is the spectral radius of $T_{1n,n}$, from Theorem \ref{thm:inertia}, all the $Z_j <0$ and when root $u$ is processed, we have
\begin{equation}\label{eq:root}
 1+2s^2-\lambda-2\frac{s^2 }{Z_{n-1}}-\frac{s^2}{Z_1} =0.
\end{equation}
\end{example}

Understanding the recurrence relation given by sequence $Z_j$ of the above example is crucial for this paper and we explain it in the sequel.

\bigskip

\subsection{Rational recurrences} \label{sec:recurrences}  The Algorithm Diagonalize, when applied to trees having pendant paths naturally produces a rational recurrence which is extremely useful to deal with eigenvalue location tasks. Namely,  we have the recurrence $x_{j}= \varphi(x_{j-1})$, where $\varphi(t)=\alpha+\frac{\gamma}{t}, \; t \neq 0$ is a rational function. A graphical  illustration for  $\varphi(t)=-3 -\frac{2}{t}, \; t \neq 0$  is given by Figure~\ref{fig:exp_neg_neg}.
\begin{figure}[H]
	\centering
	\includegraphics[width=6cm]{ 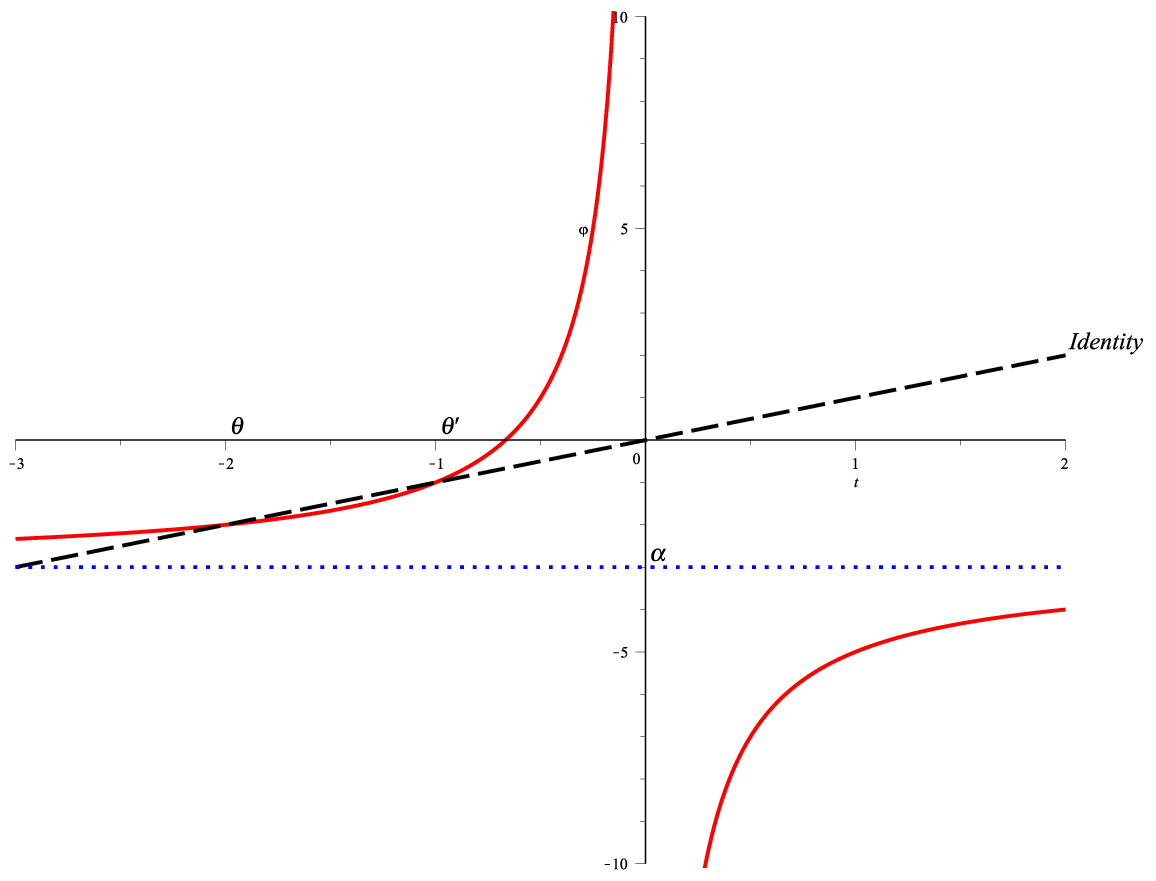}
	\caption{ Rational function $\varphi(t)=-3 -\frac{2}{t}, \; t \neq 0$.}\label{fig:exp_neg_neg}
\end{figure}
Here,  $\alpha<0$ and $\gamma<0$ produces two fixed points $\theta$ and $\theta'$, because $\Delta=\alpha^2+ 4\gamma >0$. In this case, $\theta= \frac{\alpha}{2} - \frac{1}{2}\sqrt{\alpha^2+ 4\gamma} $ and $\theta'= \frac{\alpha}{2} +\frac{1}{2}\sqrt{\alpha^2+ 4\gamma} $, so that $\theta < \theta'<0$, because $\theta \theta'=-\gamma>0$.

Since we are dealing with a rational recurrence we must require that the initial condition $x_1$ be such that $x_{j}\neq 0$ for all $j \geq 1$. Thus, when we are talking about solutions of the rational recurrence we are implicitly assuming that the initial condition satisfy $x_1 \not\in \{0\} \cup \mathcal{Z}$, where
\[\mathcal{Z}:=\{y \neq 0\, | \,   \exists \, k\geq 1\, s.t.\, \varphi^k(y)=0\} = \bigcup_{k\geq 1} \varphi^{-k}(0).\]

\begin{definition}
	The set $\mathcal{Z}$ is called the null set for the rational recurrence $x_{j+1}= \varphi(x_{j})$.
\end{definition}

We recall the following theorem from \cite{OliveTrevAppDiff} (see also the extended version \cite{OliveTrevAppDiffArXiv}).
\begin{theorem} \cite[Proposition 3]{OliveTrevAppDiffArXiv} \label{thm: dyn neg neg Delta posit} Let  $\alpha<0$ and  $\gamma<0$ be fixed numbers such that  $\Delta=\alpha^2+ 4\gamma >0$ and let $\varphi(t)=\alpha+\frac{\gamma}{t}, t\neq 0$. Define $I=(-\infty, -\sqrt{-\gamma})$, $I'=(-\sqrt{-\gamma}, 0)$, $\theta=  \frac{\alpha}{2} - \frac{1}{2}\sqrt{\alpha^2+ 4\gamma}$ and $\theta'= \frac{\alpha}{2} + \frac{1}{2}\sqrt{\alpha^2+ 4\gamma}$. Then, the recurrence $x_{j}= \varphi(x_{j-1})$ satisfy
	\begin{enumerate}
		\item  $\theta$ is an attracting fixed point of $\varphi(t)$ for  $t\in I$, while $\theta'$ is a repelling fixed point for  $t \in I'$.
		\item  The null set $\mathcal{Z}$ is a decreasing sequence $(c_j)_{j\in \mathbb{N}}$ contained in $(\theta', -\frac{\gamma}{\alpha}]$, with $c_1=-\frac{\gamma}{\alpha}$ and $\lim_{j\to \infty} c_j =\theta'$, 
		\item If $x_1 \in (-\infty, \theta)$, then $x_j$ increases and converges to $\theta$ for $j \geq 1$. If $x_1 \in (\theta,\theta')$, then $x_j$ decreases and converges to $\theta$ for $j \geq 1$.
		\item If $x_1 \in (\theta', 0)$, then there exists $m \in \mathbb{N}$ such that $x_j$ increases for $j=1,\ldots, m-1$ and $x_m>0$. Then, $x_{m+1}<\theta$ and $x_j$ increases and converges to $\theta$ for $j \geq m+1$. If $x_1=\theta$ (resp. $x_1=\theta'$), then $x_j=\theta$ (resp. $x_j=\theta'$), for $j \geq 1$.
	\end{enumerate}
\end{theorem}

We observe that recurrence obtained in Example \ref{ex:T1nn} given by

	\begin{equation}
		\begin{cases}
		Z_1 & =1-\lambda \\
		Z_j & =\varphi(Z_{j-1}), \text{ for } 2 \leq j \leq n-1,
	\end{cases}
	\end{equation}\label{eq:rec_fund}
falls in this frame work with $\varphi(t)=1+s^2-\lambda -\frac{s^2}{t}$, $t\neq 0$,  $\alpha:=1+s^2-\lambda <0$ and $\gamma:=-s^2 <0$. It satisfies the hypothesis of Theorem~\ref{thm: dyn neg neg Delta posit} and, in particular, $Z_1 \not\in \mathcal{Z}$, if  $\Delta=\alpha^2+ 4\gamma >0$, or equivalently $$\lambda > (1+|s|)^2,$$
which motivates the following notation for $s$.
\begin{definition}\label{def:adapted}
   Given $\lambda>1$ a real number, we say that $s$ is adapted to $\lambda$ if $\lambda>(1+|s|)^2$ (or equivalently $1- \sqrt{\lambda}< s< \sqrt{\lambda} -1$).
\end{definition}

\subsection{Some spectral properties of the deformed Laplacian Matrix}\label{sec:Spectrum of deformed Laplacian associated to a tree}

In reference \cite{DOT2025}, a few properties of the spectrum of the deformed Laplacian $M_G(s)$ were obtained. This section is devoted to recover some of them, particularly for the case when $G$ is a tree.

\begin{lemma}[\cite{DOT2025}]  \label{lem:initial lower bound}
Let $G$ be a simple, undirected and connected graph of order $n$ with maximum degree $\Delta$. If $s\in (-\infty, 0]\cup [1,+\infty)$, then
\begin{align*}
\rho(M_{G}(s))\geq \frac{1}{2}\left(s^{2}(\Delta-1)+2+\vert s\vert\sqrt{s^{2}(\Delta-1)^{2}+4\Delta}\right).
\end{align*}
The equality holds if and only if $G\cong K_{1,n-1}$.
\end{lemma}

\begin{lemma}[\cite{DOT2025}] \label{lem:bound}
Let $G$ be a simple, undirected and connected graph of order $n$ with maximum degree $\Delta$. Let $M_G(s)$ be its deformed Laplacian matrix. Then,
	\begin{align*}
		\rho(M_{G}(s))\leq 1+s^2(\Delta-1)+|s|\rho(A(G)).
	\end{align*}
\end{lemma}

\begin{proposition}[\cite{DOT2025}] \label{prop:properties for trees}
Let $T\neq K_1$ be a tree and let $M_T(s)$ be its deformed Laplacian matrix. Then
\begin{itemize}
    \item[$(1)$] $0\in Spect(M_T(s))$ if and only if $s=\pm 1$;
    \item[$(2)$] $M_T(s)$ is positive definite if, and only if, $|s|<1$;
    \item[$(3)$] $\rho(M_T(s))>1$;
    \item[$(4)$] If $T$ has a pendant path $P_2$ then $\rho(M_T(s))>1+s^2$;
    \item[$(5)$] Let $v$ be a pendant vertex of a tree $T$. If $T'=T\setminus \{v\}$ then $\rho(M_{T'}(s))<\rho(M_{T}(s))$;
    \item[$(6)$] If $\Delta(T)\geq 4$, then $\rho(M_{T}(s))>(1+|s|)^2$ (analogously, $\rho(M_{T}(s))>1+\sqrt{3}|s|+s^2$ if $\Delta(T)\geq 3$);
    \item[$(7)$] If $T$ is a starlike tree $T:=[q_1,\ldots,q_k]$ then $\rho(M_{T}(s))\leq 1+s^2(\Delta(T)-1) +|s|\frac{k}{\sqrt{k-1}}$.
\end{itemize}
\end{proposition}

We remark that Proposition~\ref{prop:properties for trees}~(5)  shows that for every value of $s \in \mathbb{R}$, the spectral radius of sub-trees does not increase, a result that is not true for general graphs (see \cite{DOT2025}).

\begin{corollary}\label{cor:s adapted lambda}
	     Let $T\neq P_n$ be a tree ($\Delta(T)\geq 3$). Then
	     \begin{enumerate}
	     	\item[$(1)$] If $s$ is super-Laplacian then it is adapted to $\lambda$ for any $\lambda \geq \rho(M_T(s))$;
	     	\item[$(2)$] If $s$ is sub-Laplacian and $\Delta(T)\geq 4$ then it is adapted to $\lambda$ for any $\lambda \geq \rho(M_T(s))$;
	     	\item[$(3)$] If $s\neq 0$ is sub-Laplacian, $\Delta(T) = 3$ and contains $T_{1,4,4}$ as a subgraph  then it is adapted to $\lambda$ for any $\lambda \geq \rho(M_T(s))$
	     \end{enumerate}
\end{corollary}
\begin{proof}
	 (1) Since $\Delta(T)\geq 3$ it contains $K_{1,3}$ as a subgraph. From Proposition~\ref{prop:properties for trees}(5) and Lemma \ref{lem:initial lower bound}, we obtain $\rho(M_T(s)) \geq \rho(M_{K_{1,3}}(s)) =  \frac{1}{2}\left(s^{2}(\Delta-1)+2+\vert s\vert\sqrt{s^{2}(\Delta-1)^{2}+4\Delta}\right)> 1+ s^2 + 2|s|,$ because $s^2 >1$.
	
	 (2) Follows directly from Proposition~\ref{prop:properties for trees}~(6).
	
	 (3)  From Proposition~\ref{prop:properties for trees}~(5), we obtain $\rho(M_T(s)) \geq \rho(M_{T_{1,4,4}}(s)) $. We claim that $\rho(M_{T_{1,4,4}}(s)) \geq 1+ s^2 + 2|s|= (1+|s|)^2$ with equality only if $s=0$. We perform  $Diag(M_{T_{1,4,4}}(s),-\lambda)$ for $\lambda=(1+|s|)^2$. We aim to find a positive output. At each leaf we initialize $Z_1= 1-\lambda$. If $Z_1>0$ we are done ($\rho(M_{T_{1,4,4}}(s)) > (1+|s|)^2$), otherwise if $Z_1=0$ then we keep these values and produce values $+2$ (positive) and $-\frac{s^2}{2}$ meaning that $\rho(M_{T_{1,4,4}}(s)) > (1+|s|)^2$. Finally, if  $Z_1<0$ we compute $Z_2$. Proceeding in the same fashion we will, in the wort case arrive to the root with $Z_1, Z_2, Z_3<0$.
	
	  Then, we compute the last value at $u$ by the formula
	 \begin{equation}\label{eq: value root 2}
	 	f(s):=1+2s^2-(1+|s|)^2 - 2\frac{s^2}{Z_{4-1}}-\frac{s^2}{1- (1+|s|)^2}
	 \end{equation}
	
	 We now consider the two cases $|s|=s$ and $|s|=-s$ and analyze the graphs   $f(s)$ which should be positive in the respective domains.
	
	 If $|s|=-s$, then
	 \[f(s)={\frac {{s}^{2} \left( 3\,{s}^{2}-8\,s+5 \right) }{ \left( 3\,s-4\right)  \left( s-2 \right) }}
	 \]
	 and, if $|s|=s$, then
	 \[g(s)= {\frac {{s}^{2} \left( 3\,{s}^{2}+8\,s+5 \right) }{ \left( 3\,s+4	\right)  \left( s+2 \right) }}
	 \]
	 whose graphs are given by Figure~\ref{fig:graphs_fg}:
	 \begin{figure}[H]
	 	\centering
	 	\includegraphics[width=4cm]{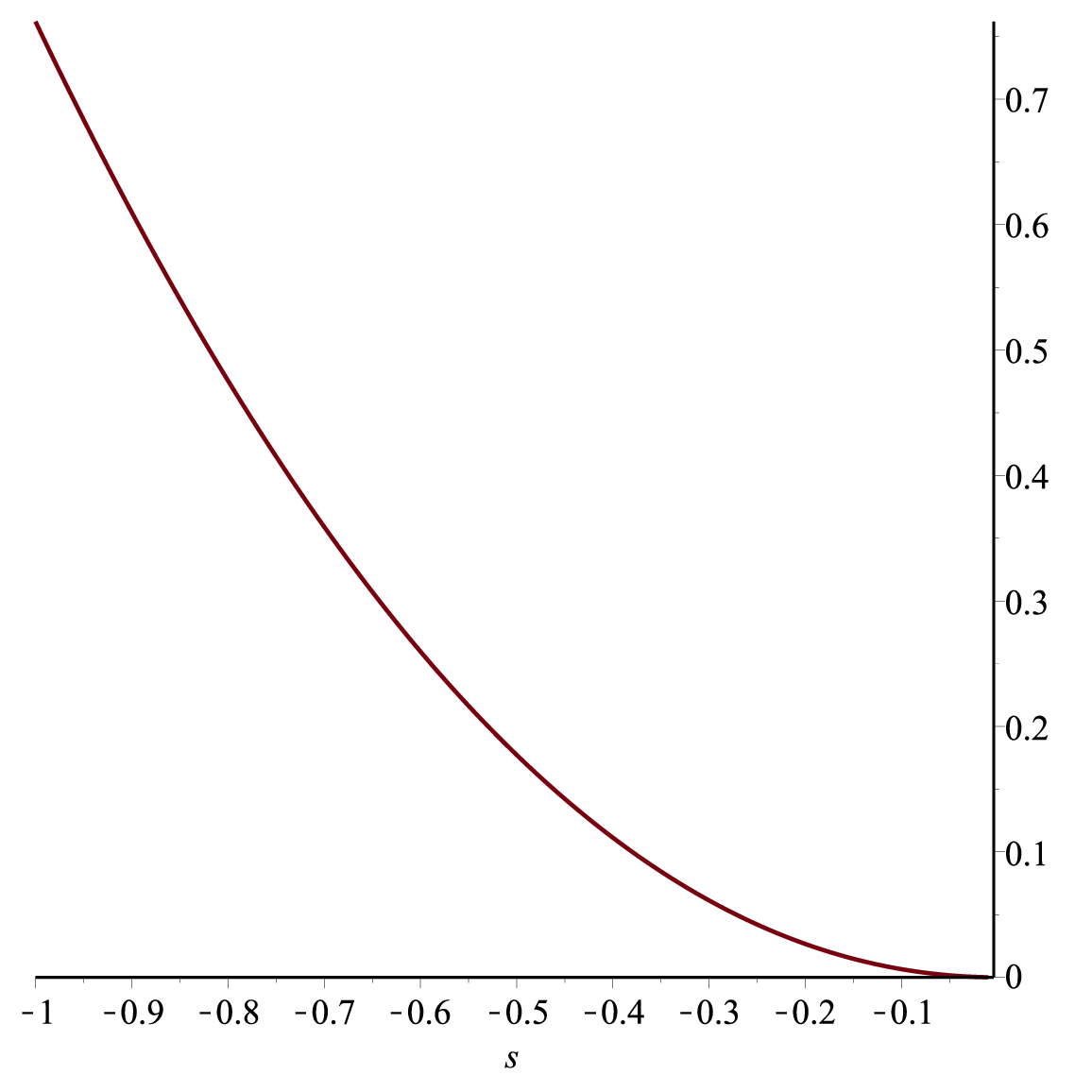}\quad	\includegraphics[width=4cm]{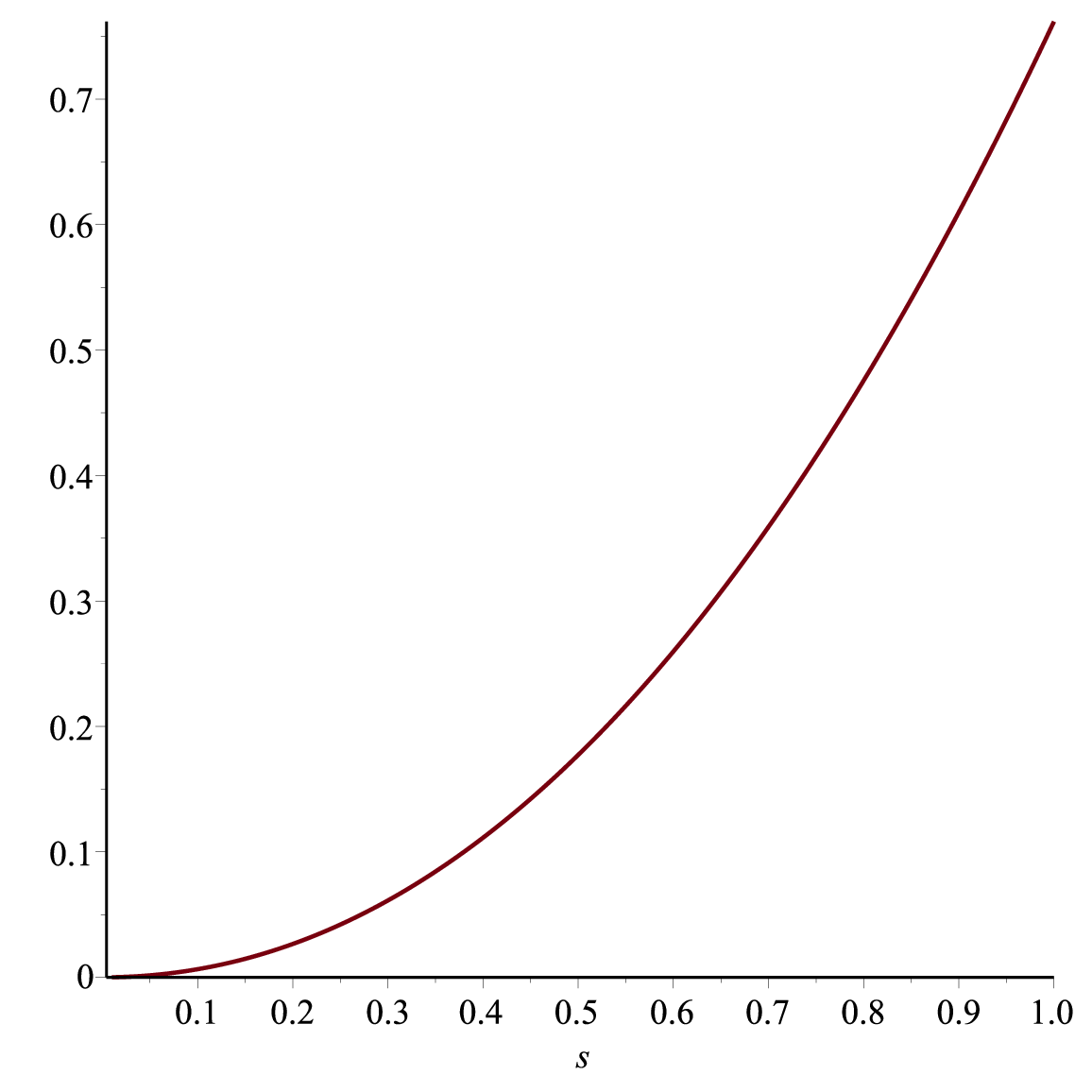}
	 	\caption{Graph of $f$ (left) and $g$(right).}\label{fig:graphs_fg}
	 \end{figure}
	 We can easily see that $f,g>0$ if $s\neq 0$. Thus,  $\rho(M_{T_{1,4,4}}(s)) > (1+|s|)^2$.
\end{proof}

\begin{remark}
	The only cases not known in Corollary~\ref{cor:s adapted lambda} are $|s|<1$, $\Delta(T) = 3$ and $T$ does not contain $T_{1,4,4}$. Since $T$ has a vertex of degree three, it must contain at most $T_{1,1,n}$, $T_{1,2,n}$, $T_{2,2,n}$, etc. That is, only one pendant path could be larger than or equal to 4 and the others are some combination of $P_2$ and $P_3$. In these cases $s$ could be non-adapted to $\lambda \geq \rho(M_{T}(s))$.
\end{remark}

\section{Deformed Laplacian limit points for a simple family}\label{sec: simple family}

In this section, we show that a single family of graphs provide $s$-deformed limit points for all $\lambda > 1$. We summarize the main result of this section, Theorem~\ref{thm:all1}, which we restate it for convenience.

\begin{theorem}\label{thm:all} Let $1 < \lambda \in \mathbb{R}$ be a real number. There exists an $s>0$ such that $\lambda$ is an $s$-deformed limit point. Precisely, $\lambda = \lim_{n \to \infty} \rho(M_{T_{1,n,n}}(s))$, for some $s\in \mathbb{R}$.
\end{theorem}

From Proposition~\ref{prop:properties for trees}~(3), we know that $\rho(M_T(s))>1$ for any tree $T$ . Thus, any limit point obtained by trees will be greater than or equal to 1. Our aim is to understand the limit points of the simple trees, the  graphs $T_{1,n,n}$ given in Example~\ref{ex:T1nn}.

We start by proving the following result.

\begin{lemma}
\label{lem:limT1nn} Let $s \in \mathbb{R}$ and $\lambda_n(s)=\rho(M_{T_{1,n,n}}(s))$ for $n\geq 4$. There exists $$\ \lim_{n\rightarrow \infty} \lambda_n(s):=\tau_0(s).$$
    \end{lemma}

\begin{proof}
Follows from the discussion in Section~\ref{sec:recurrences} that the application of diagonalize algorithm of Figure~\ref{fig:diagonalize} to $T_{1,n,n}(s)$ with $\lambda_n=\rho(M_{T_{1,n,n}}(s))$ gives rise to the recurrence $(Z_j)_{j \in \mathbb{N}}$ which is the recurrence from Theorem~\ref{thm: dyn neg neg Delta posit} for $\alpha_n:=1+s^2-\lambda_n <0$ and $\gamma=-s^2 <0$. Consequently,  we have two distinct solutions,
\begin{equation}\label{ineq}
	\Delta_n(s):=(1+s^2-\lambda_n)^2-4s^2>0
\end{equation} (here $\Delta_n$ is not the maximum degree, but the quantity defined in \cite{OliveTrevAppDiff}) or, equivalently, $1-\sqrt{\lambda_n}<s<-1+\sqrt{\lambda_n}$, also equivalent to $\lambda_n> (1+|s|)^2$. The solutions, in this case, are given by
\begin{align}
	\theta_n(s) & :=\frac{(1+s^2-\lambda)-\sqrt{(1+s^2-\lambda_n)^2-4s^2}}{2} \quad \text{ and } \\
	\theta'_n(s) & :=\frac{(1+s^2-\lambda)+\sqrt{(1+s^2-\lambda_n)^2-4s^2}}{2}.
\end{align}
Notice that $\theta_n(s)\cdot \theta'_n(s)=s^2$.

It is obvious that $Z_1=1-\lambda_n < \theta_n$. Indeed, we have $1-\lambda_n<\theta_n$ if, and only if,
\[ 2-2\lambda_n < 1+s^2-\lambda_n-\sqrt{(1+s^2-\lambda_n)^2-4s^2} \]
if, and only if,
\[ \sqrt{(1+s^2-\lambda_n)^2-4s^2}<-1+s^2+\lambda_n,\] which happens if and only if $-2\lambda_n s^2<2\lambda_n s^2$, which is true since $\lambda_n>1$. Since $s$ is adapted to $\lambda_n$, we observe that all $Z_j$ are smaller than $\theta_n(s)$ because $Z_j \to \theta_n(s)$, increasingly, when $j \to \infty$ (see Theorem~\ref{thm: dyn neg neg Delta posit}).

The sequence $(\lambda_n)_{n \in \mathbb{N}}$ is strictly increasing because $T_{1,n,n}$  is obtained from  $T_{1,n+1,n+1}$ by deletion of two edges (see Proposition~\ref{prop:properties for trees}).  Also, from Lemma~\ref{lem:bound}, as $T_{1,n,n}$ has  maximum degree $\Delta=3$ (and we can bound $\rho_A(T_{1,n,n})$ by $2\sqrt{\Delta-1}$ from \cite{BoundTree})
\[\lambda_n=\rho(M_{T_{1,n,n}}(s))\leq 1+s^2(\Delta(T_{1,n,n})-1)+|s|\rho(A(T_{1,n,n})) < \infty, \forall n \geq 2.\]
From the classical result of real analysis we obtain that $(\lambda_n)_{n \in \mathbb{N}}$ increasing and bounded, hence a convergent sequence, producing an $s$-deformed Laplacian limit point $\tau_0(s)$.

\end{proof}

In the remaining of this section, we discuss the distribution of $\tau_0(s)$, reasoning the truth of Theorem~\ref{thm:all}.
In order to obtain a formula for the $s$-deformed Laplacian limit points $\tau_{0}(s)$ we recall that  from Equation~\ref{eq:root} we know that $1+2s^2-\lambda_n - 2\frac{s^2}{Z_{n-1}}-\frac{s^2}{Z_1} = 0$  is equivalent to
\begin{equation}\label{eq: value root 3}
	1+2s^2-\lambda_n - 2\frac{s^2}{(Z_{n-1} - \theta_n) + \theta_n}-\frac{s^2}{1-\lambda_n} = 0.
\end{equation}

Using the fact that   $\theta_{n}(s) =\frac{(1+s^2-\lambda_n)-\sqrt{(1+s^2-\lambda_n)^2-4s^2}}{2} $ is continuous in $\tau_{0}(s)$ we obtain
\[\lim_{n\to \infty} \theta_{n}(s) =\frac{(1+s^2-\tau_{0}(s))-\sqrt{(1+s^2-\tau_{0}(s))^2-4s^2}}{2}.\]

Taking the limit in Equation~\ref{eq: value root 3} we obtain,
\begin{equation*}
	1+2s^2-\tau_{0}(s) - 2\frac{s^2}{\lim_{n\to \infty} \theta_{n}(s)}-\frac{s^2}{1-\tau_{0}(s)} = 0,
\end{equation*}
because $Z_{n-1} - \theta_n \to 0$ when $n \to \infty$. As $\theta_n(s)\cdot \theta'_n(s)=s^2$ we obtain the equivalent form
\[
1+2s^2-\tau_{0}(s) - 2\frac{s^2}{\frac{(1+s^2-\tau_{0}(s))-\sqrt{(1+s^2-\tau_{0}(s))^2-4s^2}}{2}}- \frac{s^2}{1-\tau_{0}(s)} = 0
\]
\[
1+2s^2-\tau_{0}(s) - \left( (1+s^2-\tau_{0}(s))+\sqrt{(1+s^2-\tau_{0}(s))^2-4s^2}\right) -\frac{s^2}{1-\tau_{0}(s)} = 0
\]
\[
s^2 -\sqrt{(1+s^2-\tau_{0}(s))^2-4s^2} -\frac{s^2}{1-\tau_{0}(s)} =0
\]
\[
 \sqrt{(1+s^2-\tau_{0}(s))^2-4s^2}=s^2 \frac{\tau_{0}(s)}{\tau_{0}(s)-1}
\]
\[
(1+s^2-\tau_{0}(s))^2-4s^2 =s^4 \left( 1+\frac{1}{\tau_{0}(s)-1}\right) ^2 .
\]
We notice that, for each $s>0$ the function
\begin{equation}\label{eq:h(t)}
 h(t):=(1+s^2-t)^2-4s^2  - s^4 \left( 1+\frac{1}{t-1}\right)^2
\end{equation}
has a vertical asymptote in $t=1$, and for $t>1$ it is differentiable with
\[h'(t)=- 2(1+s^2-t)   - 2 s^4 \left( 1+\frac{1}{t-1}\right) \left( -\frac{1}{(t-1)^2}\right) >0. \]
Since $\lim_{t \to 1^+} h(t) = -\infty$ and $\lim_{t \to  +\infty} h(t) =+\infty$ we have a unique root $t=\tau_{0}(s) \in (1, +\infty)$.

The actual value $\tau_{0}(s)$ can be computed by standard numerical methods. Otherwise, as $\frac{\tau_{0}(s)}{\tau_{0}(s)-1} >1$, from the sequence of equations above,  we can always reduce the root search  to the degree four polynomial
\begin{equation}\label{eq: solve tau0}
p(t)={t}^{4}+ \left( -2\,{s}^{2}-4 \right) {t}^{3}+ \left( 2\,{s}^{2}+6
 \right) {t}^{2}+ \left( -2\,{s}^{4}+2\,{s}^{2}-4 \right) t+{s}^{4}-2
\,{s}^{2}+1,
\end{equation}
whose unique solution in the interval $(1, +\infty)$ is $\tau_0(s)$.

It is worth observing that Equation~\eqref{eq: solve tau0} gives the same result for $s$ or $-s$, what is expected from the fact that $\sigma(M_G(s))=\sigma(M_G(-s))$ for bipartite graphs. Moreover, when $s=1$ or $s=-1$, we retrieve the Laplacian limit point for $T_{1,n,n}$, a result due to Guo \cite{guo2008}.
\begin{corollary}\label{cor:lap} Let $L=M_{T_{1,n,n}}(1)$ be the Laplacian matrix of $T_{1,n,n}$. Let $\lambda= \lim_{n\to \infty} \rho(L)$. Then $$\lambda= \frac{1}{3}\,\sqrt [3]{54+6\,\sqrt {33}}+4\,{\frac {1}{\sqrt [3]{54+6\,\sqrt {33}}}}+2\approx 4.38+$$
\end{corollary}

A table showing the values of $\tau_{0}(s)$ numerically obtained is:
\begin{center}
	\begin{tabular}{ll|ll}
		$s$ &	$\tau_{0}(s)$  &  $s$ &	$\tau_{0}(s)$\\
		\hline
		0.001 & 1.002059342 & 0.5&	 2.341081806\\
		0.01& 1.020698941 & 0.6&	 2.689803637\\
		0.1&	 1.217675873 & 0.7&	 3.067507378\\
		0.2&	 1.459682287 & 0.8&	 3.475060020\\
		0.3&	 1.726955383 & 0.9&	 3.913288615\\
		0.4&	 2.020441181 & 1.0&	 4.382975768\\
		5 &  53.36963067 & 		10 & 203.4647577\\
	\end{tabular}
\end{center}
The negative values of $s$ provide the same output, as expected (see Section~\ref{sec:PF}) and we can see that all the values  $\lambda>1$  are $s$-deformed Laplacian limit points (for different values of $|s|$, evidently).

From Proposition~\ref{prop:properties for trees}, any small ($\lambda<1$) $s$-deformed Laplacian limit point (if there are any) should be obtained from non-tree graphs or from trees without $P_2$ (otherwise $\lambda \geq1$).

Moreover, since the range of $s \mapsto \tau_{0}(s)$ is $(1,+\infty)$ we can expect to find $s$-deformed Laplacian limit points for a fixed $s$ in this interval. This in fact the theme of the next section. For a fixed $s$, we want to characterize which real numbers are $s$-deformed limit points.

\section{Deformed Laplacian limit points for caterpillars}\label{sec: Deformed Laplacian limit points for caterpillars}

In this section, inspired by the original work of  Shearer \cite{shearer1989distribution} (and also \cite{AalphaOlivTrev}), we aim to characterize, for a fixed $s$, big $s$-deformed Laplacian limit points $\lambda$  (that is, $\lambda \in \mathcal{L}(s)$ such that $\lambda \geq 1$), where $s$ is adapted to $\lambda$, that is  $\lambda > (1+|s|)^2$ or equivalently, $1-\sqrt{\lambda}< s < \sqrt{\lambda} -1$. Our choice is a set of trees (which contains $T_{1,n,n}$) named caterpillars: trees such that,  after removing all the leafs, remains only a path called backbone. As we are aiming trees (which are bi-partite graphs) we can just  look at positive values of $s$, that is, $0< s < \sqrt{\lambda} -1$.

A caterpillar denoted by $T_k=[r_1,r_2, \ldots, r_k], \; k \geq 2$ has a path with $k$ vertices $v_i, \; i \in \{1,\ldots k\}$ each one with $r_i$ pendant vertices (see Figure~\ref{fig:caterpil}). The vertices $v_i$ are called {\it back nodes}. The first task is to build a sequence of caterpillars whose deformed Laplacian matrices have spectral radius strictly smaller than a fixed number.
\begin{theorem}\label{thm: condition caterp}
Given $\lambda$ and $s$ real fixed numbers, let ${\displaystyle \delta:=\frac{s^2 \lambda }{\lambda -1}}$. A caterpillar $T_k=[r_1,r_2, \ldots, r_k], \; k \geq 2$ satisfies $\rho(M_{T_{k}}(s))< \lambda$ if, and only if,
	\[
	\begin{cases}
		b_1 :=1 - \lambda  +r_{1} \delta <0,\\
		b_{j} := \varphi(b_{j-1})  +r_{j} \delta <0,\; 2 \leq j \leq k-1,\\
		b_{k} := -s^2 + \varphi(b_{k-1})  +r_{k} \delta <0,
	\end{cases}\]
	 where $\varphi(t):= 1+s^2 - \lambda-  \frac{ s^2}{t}, \; t \neq 0$.
\end{theorem}
\begin{proof}
Consider the execution of $\mbox{Diagonalize}(M_{T_{k}}(s), - \lambda)$, where the caterpillar $T_k:=T_k(\lambda)=[r_1,r_2, \ldots, r_k], \; k \geq 2$. In Figure~\ref{fig:caterpil} we can visualize the initial values at each vertex.
	\begin{figure}[H]
		\centering
		\includegraphics[width=0.5\linewidth]{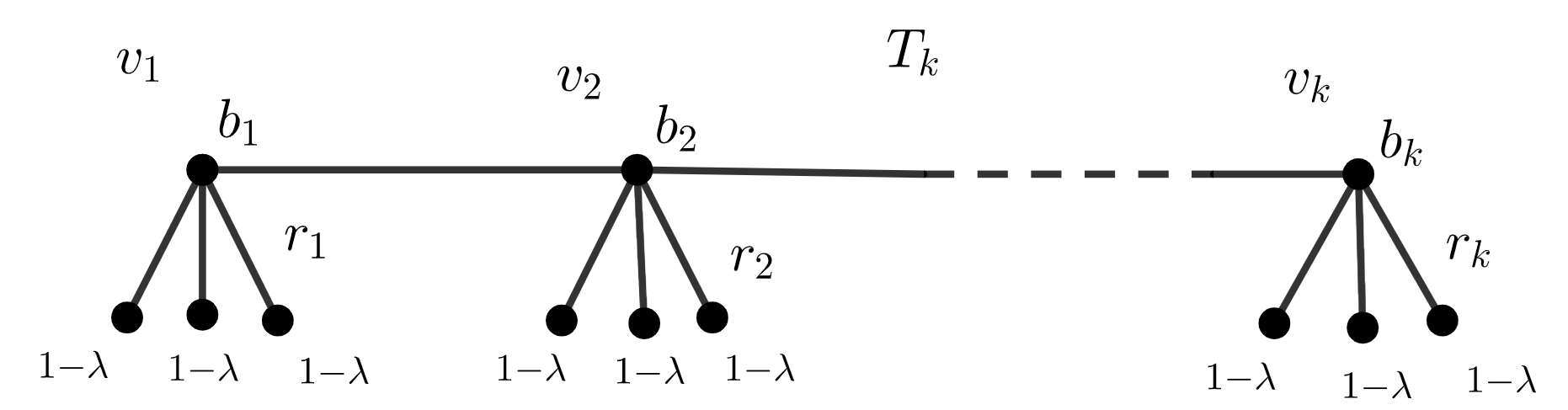}
		\caption{Caterpillar $T_k=[r_1,r_2, \ldots, r_k], \; k \geq 2$.}
		\label{fig:caterpil}
	\end{figure}
	We aim to obtain $\rho(M_{T_{k}}(s))< \lambda$, thus from Theorem~\ref{thm:inertia}, each output must be negative.	 The vertex $v_1$ is initialized with $1+(r_{1}+1 -1)s^2 - \lambda=1+r_{1} s^2 - \lambda $, the next ones are initialized with	 $1+(r_{j}+2 -1)s^2 - \lambda = 1+(r_{j}+1 )s^2 - \lambda $ and the last one, $v_k$ is $1+(r_{k}+1 -1)s^2 - \lambda=1+r_{k} s^2 - \lambda $. At each leaf, we start with the value $1+(1 -1)s^2 - \lambda=1 - \lambda <0 $ because we have a tree different of $K_1$.
	
Now we compute the remaining outputs according to the algorithm.
	 \[b_1:= 1+r_{1} s^2 - \lambda - \frac{r_{1} s^2}{1 - \lambda} = 1 - \lambda  +r_{1}s^2 \frac{\lambda }{\lambda -1}.\]
	 \[b_{j}:= 1+(r_{j}+1 )s^2 - \lambda -  \frac{ s^2}{b_{j-1}} - \frac{r_{j} s^2}{1 - \lambda} = 1+s^2 - \lambda-  \frac{ s^2}{b_{j-1}}   +r_{j}s^2 \frac{\lambda }{\lambda -1},\; 2 \leq j \leq k-1.\]
	 \[b_{k}:= 1+r_{k} s^2 - \lambda -  \frac{ s^2}{b_{k-1}} - \frac{r_{k} s^2}{1 - \lambda} = 1 - \lambda-  \frac{ s^2}{b_{k-1}}   +r_{k} s^2 \frac{\lambda }{\lambda -1}.\]
	  Introducing the notation ${\displaystyle \delta:=\frac{s^2 \lambda }{\lambda -1}}$  and recalling that
	 $\varphi(t):= 1+s^2 - \lambda-  \frac{ s^2}{t}, \; t \neq 0$, we obtain the desired formulas.
	
\end{proof}

We recall once more here that $\varphi(t)= 1+s^2 - \lambda-  \frac{ s^2}{t}, \; t \neq 0$  has two fixed points (see Theorem~\ref{thm: dyn neg neg Delta posit} and discussion in Section~\ref{sec:recurrences})
\begin{align}
	\theta(s) & :=\frac{(1+s^2-\lambda)-\sqrt{(1+s^2-\lambda)^2-4s^2}}{2} \quad \text{ and } \\
	\theta'(s) & :=\frac{(1+s^2-\lambda)+\sqrt{(1+s^2-\lambda)^2-4s^2}}{2},
\end{align}
because  we assumed that $s$ is adapted to $\lambda$, or equivalently, $\lambda> (1+|s|)^2$.

\subsection{Shearer's sequences}
Given a fixed value $\lambda>1$, we aim to find a sequence of caterpillars $(T_k)_{k \in \mathbb{N}}$, where $T_k:=T_k(\lambda)$  such that $\lim_{k\to \infty} \rho(M_{T_{k}}(s))= \lambda$, where $s$ is adapted to $\lambda$. By our choice, using the asymptotic behavior of the recurrence defined by $\varphi$ (see \cite{OliveTrevAppDiff}  or Theorem~\ref{thm: dyn neg neg Delta posit} for details), we shall select caterpillars according to Shearer's approach. We notice that one must require $b_{j}<\theta'(s)$ (it is not enough to choose $b_{j}<0$) in each step because  the right side of $\theta'(s)$ is a repellent point. Otherwise, at some moment, $b_{m}>0$ (see Theorem~\ref{thm: dyn neg neg Delta posit}). Moreover, in order to obtain the best approximation as possible we always choose the maximum value of $r_k$ satisfying this condition. The approach may be summarized as follows:
\begin{definition}\label{def:Shearer sequence}
	  The Shearer sequence $(T_k)_{k \in \mathbb{N}}$ is obtained by
	  \[
	  \begin{cases}
	  	r_{1} : =\big\lfloor \delta^{-1} \left( \theta'(s) +\lambda -1 \right) \big\rfloor\\
	  	r_{j} : =\big\lfloor \delta^{-1} \left( \theta'(s) - \varphi(b_{j-1}) \right) \big\rfloor,\; 2 \leq j \leq k-1\\
	  	r_{k} : =\big\lfloor \delta^{-1} \left( \theta'(s) - \varphi(b_{k-1}) +s^2\right)  \big\rfloor.\\
	  \end{cases}\]
\end{definition}

We remark that, as we choose the maximum value in Definition~\ref{def:Shearer sequence} the following equation holds
\begin{equation}\label{eq:maximal shearer}
	  \theta'(s) - \delta < b_j < \theta'(s), \; \forall j.
\end{equation}

\begin{theorem}\label{thm: shearer good definition}
	The Shearer's sequence $(T_k)_{k \in \mathbb{N}}$ in Definition~\ref{def:Shearer sequence} is well-defined, that is, $r_j \geq 0$. Moreover $\rho(M_{T_{k}}(s))< \lambda$.
\end{theorem}
\begin{proof}
	 The proof is similar to \cite{AalphaOlivTrev}, using induction on $j$.  We only need to ensure that it is possible to continue the process of choosing $r_k$ from $r_{k-1}$ once we have chosen $r_1$. Recall that the main constraint is
	 \[
	 \begin{cases}
	 	b_1 =1 - \lambda  + r_{1} \delta <\theta'(s)\\
	 	b_{j}= \varphi(b_{j-1})  +r_{j} \delta <\theta'(s),\; 2 \leq j \leq k-1\\
	 	b_{k}= -s^2 + \varphi(b_{k-1})  +r_{k} \delta <\theta'(s).
	 \end{cases}\]
	 In the worst case, we can consider  $r_{1} =0$ in such way that $1 - \lambda  + r_{1} \delta <\theta'(s)$ because of the equivalences
	 \[1 - \lambda  + 0 \; \delta <\theta'(s)\]
	 \[1 - \lambda  <\frac{(1+s^2-\lambda)+\sqrt{(1+s^2-\lambda)^2-4s^2}}{2} \]
	 \[2(1 - \lambda)  <(1+s^2-\lambda)+\sqrt{(1+s^2-\lambda)^2-4s^2} \]
	 \[(1 - \lambda)  <s^2+\sqrt{(1+s^2-\lambda)^2-4s^2}, \]
being the last one always true for $\lambda>1$. Therefore, we can choose  $r_{1}\geq 0$, the largest number satisfying $b_1 =1 - \lambda  + r_{1} \delta <\theta'(s)$.
	
	 Supposing that we have chosen  $r_{j-1}$ and produced 	 $b_{j-1}$ we must be able to choose $r_{j} \geq 0$ such that $b_{j}= \varphi(b_{j-1})  +r_{j} \delta <\theta'(s)$. This follows from the properties of the recurrence $\varphi$ because $\varphi(b_{j-1})$ will be closer to $\theta(s)$  provided that $b_{j-1}< \theta'(s)$ (see \cite{OliveTrevAppDiff}  and Theorem~\ref{thm: dyn neg neg Delta posit} for details). Therefore, we can choose  $r_{j}\geq 0$, the largest number satisfying $b_{j}= \varphi(b_{j-1})  +r_{j} \delta <\theta'(s)$ (which will be at least zero because $\varphi(b_{j-1}) < \theta'(s)$).
	
	 Analogously,  in the last step $b_{k}= -s^2 + \varphi(b_{k-1})  +r_{k} \delta <\theta'(s)$ always admits a choice $r_{k} \geq 0$ because
	 $b_{k-1}< \theta'(s)$ implies $\varphi(b_{k-1})   <\varphi(\theta'(s))=\theta'(s) $  and $-s^2 + \varphi(b_{k-1}) < \varphi(b_{k-1})   <\theta'(s)$  thus $-s^2 + \varphi(b_{k-1})  +0\, \delta <\theta'(s)$.	 This proves our claim since we can continue this process indefinitely.
	
	 Finally, $\rho(M_{T_{k}}(s))< \lambda$ by construction and Theorem~\ref{thm: condition caterp}, since $b_{j}< \theta'(s)$ implies  $b_{j}< 0$.
\end{proof}

A fundamental step towards to the existence of limit points is the monotonicity and boundness of the sequence $(\rho(M_{T_{k}}(s)))_{k \in \mathbb{N}}$. Boundness is guaranteed by construction. Next, we prove monotonicity.

\begin{theorem}\label{thm: shearer increasing}
Shearer's sequence $(T_k)_{k \in \mathbb{N}}$ satisfies $\rho(M_{T_{k}}(s))<  \rho(M_{T_{k+1}}(s)) < \lambda$. In particular, there exists a unique number
	\[\hat \lambda= \lim_{k \to \infty}  \rho(M_{T_{k}}(s)) \leq  \lambda.\]
\end{theorem}
\begin{proof}
	 	 Consider
	 \[
	 \begin{cases}
	 	r_{1} : =\big\lfloor \delta^{-1} \left( \theta'(s) +\lambda -1 \right) \big\rfloor\\
	 	r_{j} : =\big\lfloor \delta^{-1} \left( \theta'(s) - \varphi(b_{j-1}) \right) \big\rfloor,\; 2 \leq j \leq k-1\\
	 	r_{k} : =\big\lfloor \delta^{-1} \left( \theta'(s) - \varphi(b_{k-1}) +s^2\right)  \big\rfloor,\\
	 \end{cases}\]
	 the sequence of degrees associated to $T_k$ and
	 \[
	 \begin{cases}
	 	\tilde r_{1} : =\big\lfloor \delta^{-1} \left( \theta'(s) +\lambda -1 \right) \big\rfloor\\
	 	\tilde r_{j} : =\big\lfloor \delta^{-1} \left( \theta'(s) - \varphi(\tilde b_{j-1}) \right) \big\rfloor,\; 2 \leq j \leq k\\
	 	\tilde  r_{k+1} : =\big\lfloor \delta^{-1} \left( \theta'(s) - \varphi(\tilde  b_{k}) +s^2\right)  \big\rfloor,\\
	 \end{cases}\]
	 the sequence of degrees associated to $T_{k+1}$. It is evident that $\tilde r_{1}= r_{1}$ and $\tilde r_{j}= r_{j},\; 2 \leq j \leq k-1$. Let us examine $\tilde r_{k}$ and $\tilde r_{k+1}$.
	
	 From Equation~\ref{eq:maximal shearer} we obtain
	 \[\theta'(s) - \delta < b_j= \varphi(b_{j-1})  +r_{j} \delta < \theta'(s), \; \forall j.\]
	 Thus,
	 \[\theta'(s) - \delta < -s^2 + \varphi(b_{k-1})  +r_{k} \delta < \theta'(s) \text{ and } \theta'(s) - \delta < \varphi(\tilde b_{k-1})  +\tilde r_{k} \delta < \theta'(s).\]
	 Since  $b_{k-1} = \tilde b_{k-1}$ we obtain the equivalent form by multiplying the second one by $-1$:
	 \[\frac{\theta'(s) - \delta+s^2 -\varphi(b_{k-1}) }{\delta} <  r_{k}  < \frac{\theta'(s) +s^2- \varphi(b_{k-1}) }{\delta} \] and \[
	 \frac{-\theta'(s) +\delta +\varphi(b_{k-1})}{\delta }  > - \tilde r_{k}  > \frac{\varphi(b_{k-1})  -\theta'(s)}{\delta }.\]
	 Adding these two inequalities we obtain,
	 \[-1< \frac{-\delta +s^2}{\delta}< r_{k} -\tilde r_{k} <\frac{\delta +s^2}{\delta}<2\]
	 \[0\leq  r_{k} -\tilde r_{k}  \leq 1\]
	 thus, $\tilde r_{k} =r_{k}$ or $\tilde r_{k} =r_{k} -1$.
	
	 If $\tilde r_{k} =r_{k}$, then $T_k$ is a subgraph of $T_{k+1}$. If $\tilde r_{k} =r_{k} -1$, then we can consider the path joining $v_k$ to $v_{k+1}$ to create a copy of $T_{k}$ inside  $T_{k+1}$. So, independently of the value of $\tilde r_{k}$, we know that $T_{k}$ is a proper subgraph of $T_{k+1}$ and by Proposition~\ref{prop:properties for trees}  we obtain that $(T_k)_{k \in \mathbb{N}}$ satisfy $\rho(M_{T_{k}}(s))<  \rho(M_{T_{k+1}}(s)) $.

From Theorem~\ref{thm: shearer increasing} we conclude that the limit of spectral radii for a Shearer's sequence $(T_k)_{k \in \mathbb{N}}$ is an $s$-deformed Laplacian limit point $\hat \lambda$ smaller than or equal to $\lambda$.

\end{proof}

In the next section we will try to establish when $\hat \lambda = \lambda$.

\section{Approximation}\label{sec: Approximation}

Proving that the Shearer's sequence of caterpillars obtained in the previous section always produces an $s$-deformed Laplacian limit point is hard, mainly because the sequences have intricate expressions and proving its convergence would be technically difficult.
In this section we will find an approximate criterion and characterize an interval of $s$-deformed Laplacian limit points for every $s \in (0,1)$.

\subsection{Approximation Criterion}\label{sec: Criterion of approximation}

\begin{theorem}\label{thm: main criterion}
	Let $(T_k)_{k \in \mathbb{N}}$ be the Shearer's sequence for $\lambda>1$ and $\varepsilon_1,\ldots,\varepsilon_k$ be the smallest positive root of $b_j(\varepsilon) $, the output of $Diagonalize(M_{T_{k}}(s), -(\lambda-\varepsilon))$ at each backnode of $T_k$. Then, $\hat \lambda= \lim_{k \to \infty}  \rho(M_{T_{k}}(s))= \lambda$ if, and only if, $\varepsilon_k \to 0$.  In particular, if  the sequence $(\beta_j)_{j \in \mathbb{N}}$, given by \eqref{eq:beta0} is unbounded, then $\varepsilon_k \to 0$ ( so that $\hat \lambda= \lambda$).
	\begin{equation}\label{eq:beta0}
		\begin{cases}
		\beta_{1}= \frac{1+ r_1 s^2 \frac{ 1 }{(\lambda -1)^2}}{ \lambda -1  - r_{1} \delta },\\
		\beta_{j} = c_{j} + \gamma_{j}\beta_{j-1},
	\end{cases}
	\end{equation}
with $c_j:= \frac{1+ r_{j} s^2 \frac{ 1 }{(\lambda -1)^2} }{  -b_{j}  }>0,\; j\geq 2$ and $\gamma_j:=\frac{ s^2}{b_{j-1} b_{j} }>0,\; j\geq 2 $ 
\end{theorem}
\begin{proof}  We notice that $Diagonalize(M_{T_{k}}(s), -(\lambda-\varepsilon))$ will have all output for the backnode vertices negative if $\rho(M_{T_{k}}(s))< \lambda$, that is, taking $\varepsilon=0$ by construction. Hence, one can obtain $|\lambda - \rho(M_{T_{k}}(s))| < \varepsilon_{k}$ if, and only if, some $b_k(\varepsilon) $ changes to a positive value in $\varepsilon=\varepsilon_k$, that is,  $\lambda -  \varepsilon_{k} < \rho(M_{T_{k}}(s)) < \lambda$. Thus, $\lim_{k \to \infty}  \rho(M_{T_{k}}(s))= \lambda$ if, and only if,   $\varepsilon_k \to 0$.
Following the reasoning of \cite{AalphaOlivTrev}, we know that $\hat \lambda= \lim_{k \to \infty}  \rho(M_{T_{k}}(s))= \lambda$ if, and only if, for any $\varepsilon>0$ there exists $k:=k_\varepsilon$ such that  $Diagonalize(M_{T_{k}}(s), -(\lambda-\varepsilon))$ has some positive output. Let us examine these numbers.

Fixed the  caterpillar $T_k=[r_1,r_2, \ldots, r_k], \; k \geq 2$, which satisfy $\rho(M_{T_{k}}(s))< \lambda$ by construction of the Shearer's sequence, we obtain
\[
\begin{cases}
	b_1(\varepsilon) :=1 - \lambda + \varepsilon  +r_{1} \delta(\varepsilon) \\
	b_{j}(\varepsilon) := \varphi_{\varepsilon}(b_{j-1}(\varepsilon) )  +r_{j} \delta(\varepsilon),\; 2 \leq j \leq k-1\\
	b_{k}(\varepsilon) := -s^2 + \varphi(b_{k-1}(\varepsilon) )  +r_{k} \delta(\varepsilon) ,
\end{cases}\]
where
\[\delta(\varepsilon):=s^2 \frac{ \lambda- \varepsilon }{\lambda- \varepsilon -1}
\text{ and }
\varphi_{\varepsilon}(t):= 1+s^2 - \lambda + \varepsilon-  \frac{ s^2}{t}, \; t \neq 0.\]

Define $\varepsilon_{0}:=\lambda -1$ then $b_1(\varepsilon)$ has a single vertical asymptote at this point and, for $\varepsilon \in (0,\varepsilon_0)$ the function $b_1(\varepsilon)$ is differentiable and
\[\frac{d\;}{d\varepsilon} b_1(\varepsilon) = 1+ r_1 s^2 \frac{ 1 }{(\lambda- \varepsilon -1)^2} \geq 1.\]
Thus, $b_1(\varepsilon)$ is monotonously increasing in $(0,\varepsilon_0)$ and continuous in $[0,\varepsilon_0)$, with $b_1(0)< \theta(s)<0$.  In particular,  there exists a unique root $\varepsilon_{1} \in (0,\varepsilon_0)$ (see Figure~\ref{fig:criteria_deformed}). Moreover, $b_1(\varepsilon)$ is twice differentiable and convex because
\[\frac{d^2\;}{d\varepsilon^2} b_1(\varepsilon) =  2 r_1 s^2 \frac{ 1 }{(\lambda- \varepsilon -1)^3} >0.\]
Using the linear approximation at $\varepsilon=0$ we obtain a root $\alpha_1= -\frac{b_1(0)}{\frac{d\;}{d\varepsilon} b_1(0)} > \varepsilon_{1}$ or equivalently
\[\beta_{1}:= \frac{1}{\alpha_1}= \frac{\frac{d\;}{d\varepsilon} b_1(0)}{-b_1(0)} = \frac{1+ r_1 s^2 \frac{ 1 }{(\lambda -1)^2}}{ \lambda -1  - r_{1} \delta }.\]

\begin{figure}[H]
	\centering
	\includegraphics[width=0.35\linewidth]{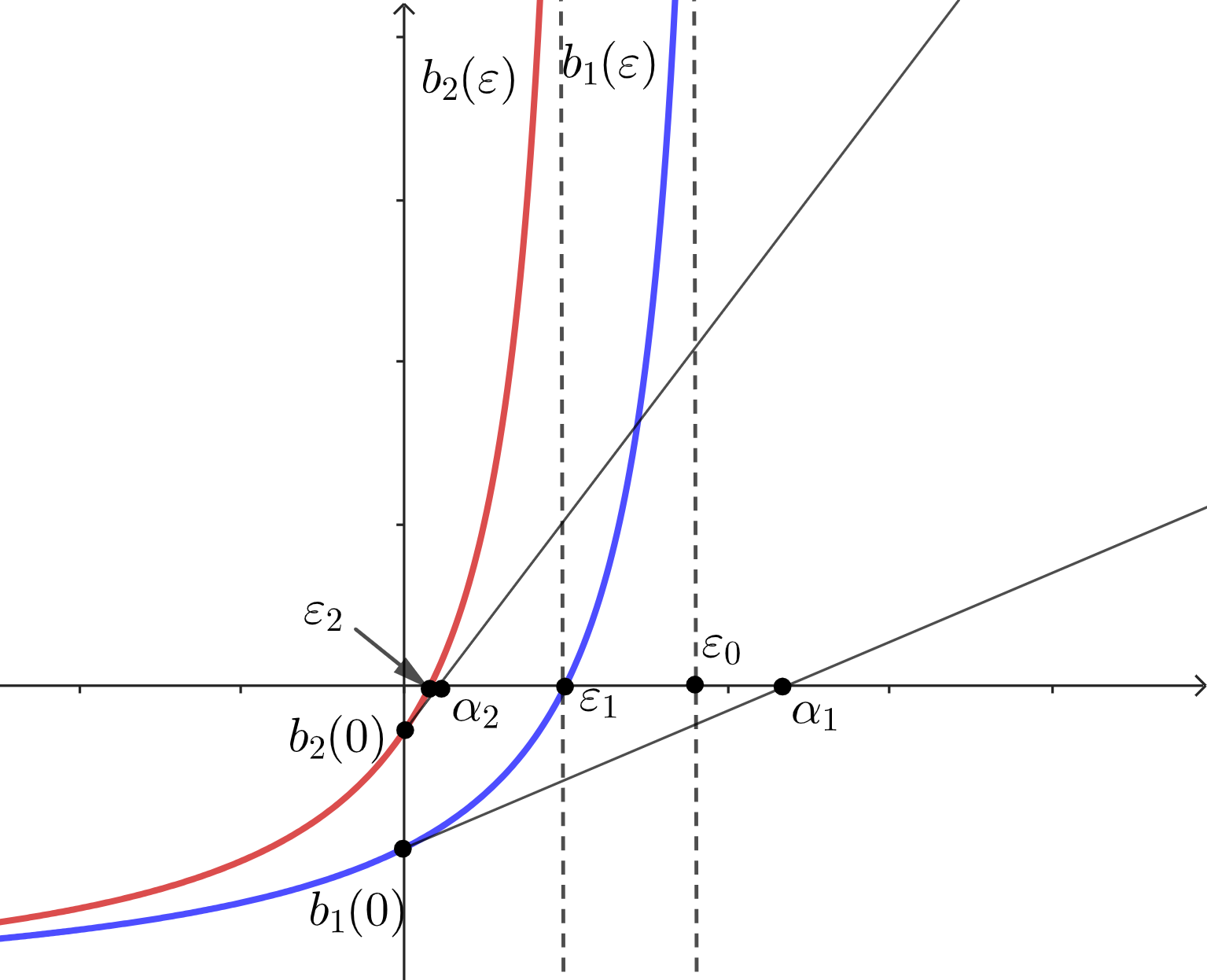}
	\caption{Sequences $\varepsilon_j$ and $\alpha_j$.}
	\label{fig:criteria_deformed}
\end{figure}

We can recursively apply this procedure at every step. $b_{j}(\varepsilon)$ has a single vertical asymptote at the point $\varepsilon_{j-1}$ and, for $\varepsilon \in (0,\varepsilon_{j-1})$ the function $b_{j}(\varepsilon)$ is differentiable and
\[\frac{d\;}{d\varepsilon} b_{j}(\varepsilon) = 1+  \frac{ s^2}{ b_{j-1}(\varepsilon)^2} \frac{d\;}{d\varepsilon} b_{j-1}(\varepsilon)  +  r_{j} s^2 \frac{ 1 }{(\lambda- \varepsilon -1)^2}   \geq 1.\]
Thus, $b_{j}(\varepsilon)$ is monotonously increasing in $(0,\varepsilon_{j-1})$ and continuous in $[0,\varepsilon_{j-1})$, with $b_{j}(0)< \theta(s)<0$.  In particular,  there exists a unique root $\varepsilon_{j} \in (0,\varepsilon_{j-1})$. Moreover, $b_{j}(\varepsilon)$ is twice differentiable and convex because
\[\frac{d^2\;}{d\varepsilon^2} b_{j}(\varepsilon) =\frac{ s^2}{ b_{j-1}(\varepsilon)^2} \frac{d^2\;}{d\varepsilon^2} b_{j-1}(\varepsilon)  -2 \frac{ s^2}{ b_{j-1}(\varepsilon)^3} \left( \frac{d\;}{d\varepsilon} b_{j-1}(\varepsilon)\right) ^2  +  2 r_{j} s^2 \frac{ 1 }{(\lambda- \varepsilon -1)^3} >0.\]

Using the linear approximation at $\varepsilon=0$ we obtain a root $\alpha_{j}= -\frac{b_{j}(0)}{\frac{d\;}{d\varepsilon} b_{j}(0)} > \varepsilon_{j}$ or equivalently
\[\beta_{j}:= \frac{1}{\alpha_j}= \frac{\frac{d\;}{d\varepsilon} b_{j}(0)}{-b_{j}(0)}
= \frac{1+  \frac{ s^2}{ b_{j-1}^2} \frac{d\;}{d\varepsilon} b_{j-1}(0)  +  r_{j} s^2 \frac{ 1 }{(\lambda -1)^2} }{  -b_{j}(0) } = \]
\[= \frac{1+ r_{j} s^2 \frac{ 1 }{(\lambda -1)^2} }{  -b_{j}(0) }  +   \frac{ \frac{ s^2}{ b_{j-1}^2} \frac{d\;}{d\varepsilon} b_{j-1}(0) }{  -b_{j}(0) }=  \frac{1+ r_{j} s^2 \frac{ 1 }{(\lambda-1)^2} }{  -b_{j}(0) }  +\frac{ s^2}{b_{j-1}(0)b_{j}(0)}  \frac{  \frac{d\;}{d\varepsilon} b_{j-1}(0)}{ -b_{j-1}}=\]
\[= \frac{1+ r_{j} s^2 \frac{ 1 }{(\lambda -1)^2} }{  -b_{j}(0) }  +\frac{ s^2}{b_{j-1}(0)b_{j}(0)} \beta_{j-1}.\]

Denoting $c_j:= \frac{1+ r_{j} s^2 \frac{ 1 }{(\lambda -1)^2} }{  -b_{j}(0) }>\frac{1 }{  -b_{j}(0) }> \frac{1 }{  \delta(s) -\theta'(s) } >0,\; j\geq 2$ and $\gamma_j:=\frac{ s^2}{b_{j-1}(0)b_{j}(0)},\; j\geq 2 $ we obtain the recurrence
\begin{equation}\label{eq:beta}
	\begin{cases}
	 \beta_{1}= \frac{1+ r_1 s^2 \frac{ 1 }{(\lambda -1)^2}}{ \lambda -1  - r_{1} \delta }\\
	 \beta_{j} = c_{j} + \gamma_{j}\beta_{j-1}
\end{cases}
\end{equation}
whose solution is
	\begin{align}
\beta_{j} &= c_{j} + \gamma_{j}(c_{j-1} + \gamma_{j-1}(c_{j-2} + \gamma_{j-2}(\cdots (c_{2} + \gamma_{2}\beta_{1}))))\nonumber\\
&=c_{j} +  \gamma_{j} c_{j-1} + \gamma_{j}\gamma_{j-1} c_{j-2} + \cdots +  \gamma_{j}\gamma_{j-1}\cdots \gamma_{2}\beta_{1}.\label{eq:explict beta}
\end{align}

In order to conclude our proof we notice that we already established that $\hat \lambda= \lim_{k \to \infty}  \rho(M_{T_{k}}(s))= \lambda$ if, and only if, $\varepsilon_k \to 0$, and $\varepsilon_k  \leq \alpha_k = \frac{1}{\beta_k}$. Hence, if the positive sequence $(\beta_{j})_{j \geq 1}  $  is unbounded then    $\varepsilon_k \to 0$ and consequently $ \lim_{k \to \infty}  \rho(M_{T_{k}}(s))= \lambda$.

\end{proof}

\subsection{Main result}

We now present our main theorem providing  intervals entirely formed by $s$-deformed Laplacian limit points for  $s \in (0,1)$. We notice that our initial goal was to characterize all $s$-deformed Laplacian limit points for a fixed value of $s$. We do not characterize all limit points because we have studied this problem for ``large'' values of  $\lambda > 1$.  However, our result is stronger in this range. In fact, we show that  any $\lambda >1$ is an $s$-deformed Laplacian limit point not for a single value of $s$ but rather for a whole interval of $s \in (0, s^*(\lambda))$, for a specific number $0 <  s^*(\lambda)<1$ .

\begin{theorem}\label{thm:suf criterion}
	 Let $(T_k)_{k \in \mathbb{N}}$ be the Shearer's sequence for $\lambda>1$. Then, there exists a unique value $0<s^*:=s^*(\lambda) <\sqrt{\lambda} -1$ given by
\[s^*(\lambda):=\left( \frac{\ell}{2}\,-{\frac {4\,{\lambda}^{2}+8\,\lambda-2}{\ell}}+{\lambda+1}\right)
\frac{\lambda-1}{3 \lambda},
 	\]
where
$\ell=\sqrt [3]{12\,\sqrt {3}\sqrt {{\frac{3\,{\lambda}^{3}+ 4\, {\lambda}^{2} +20\,\lambda-4}{\lambda}}}{\lambda}^{2}-28\,{\lambda}^{3}+24\, {\lambda}^{2}-48\,\lambda+8}$, such that for all $s \in (0,s^*)$,  $\lambda$ is an $s$-deformed Laplacian limit point. Moreover, the correspondence $\lambda \to s^*(\lambda)$ is increasing,  $\lim_{\lambda \to 1} s^*(\lambda) = 0$, and $\lim_{\lambda \to \infty} s^*(\lambda) = 1$. Consequently, for any fixed value $\lambda_0>1$ and for any $s \in (0,s^*(\lambda_0))$ the interval $[\lambda_0, \; +\infty)$ is entirely formed by $s$-deformed Laplacian limit points.
\end{theorem}
\begin{proof} We will split the proof in four parts.

 Our first claim characterizes the size of $b_j$ using a scalar function.\\
    \textbf{Claim 1:}  If the function $F(s, \lambda):= \theta'(s) -\delta+s$ defined for $\lambda>1, 0\leq s\leq \sqrt{\lambda} -1$, is positive then $\frac{ s^2}{b_{j-1} ^2 }> 1$, for all $j$.
    \vspace{0.5cm}\\
     \begin{proof}
 Indeed,  as $b_{j}<0$ (Shearer's sequence property), the inequality $\frac{ s^2}{b_{j-1} ^2 }> 1$  is equivalent to $b_{j}>-s$. On the other hand from Equation~\eqref{eq:maximal shearer}, $b_j>  \theta'(s) -\delta$, so it is enough to show that  $\theta'(s) -\delta>-s$ or equivalently, that the function $F(s, \lambda)= \theta'(s) -\delta+s$ is positive. Thus, if $F(s, \lambda):= \theta'(s) -\delta+s$ is positive then $\frac{ s^2}{b_{j-1} ^2 }> 1$, for all $j$.
    \end{proof}

    Inspired by Claim 1, we now aim to find conditions under which  $F(s, \lambda_0)>0$ holds for a fixed $\lambda_0>1$.\\
     \textbf{Claim 2:}  For any fixed $\lambda_0>1$, there exists a unique value $0<s^*=s^*(\lambda_0) <\sqrt{\lambda_0} -1$ such that  $F(s, \lambda_0)>0$ for each $0<s<s^*$.  In particular, the correspondence  $\lambda_0 \to s^*(\lambda_0)$ is increasing $\lim_{\lambda_0 \to 1} s^*(\lambda_0) = 0$ and $\lim_{\lambda_0 \to \infty} s^*(\lambda_0) = 1$.

    \begin{proof}
    	 	The inequality $F(s, \lambda)>0$ is equivalent to  the expression
    	 \[F(s, \lambda):=\left(\frac{(1+s^2-\lambda)+\sqrt{(1+s^2-\lambda)^2-4s^2}}{2}   \right) -\frac{ s^2 \lambda}{\lambda -1}  +s >0.\]
    	 We notice that, for each fixed $\lambda_0$, as $s$ is adapted to $\lambda_0$ the valid domain of $s$ in  $F(s, \lambda_0)$ is $ 0\leq s <\sqrt{\lambda_0} -1$. It is easy to see that the inequality can only hold if $s>0$ since  $F(0, \lambda_0)=0$
    	 \[\left(\frac{(1+0^2-\lambda_0)+\sqrt{(1+0^2-\lambda_0)^2 - 4 \;0^2}}{2}   \right) -\frac{ 0^2 \lambda_0}{\lambda_0 -1}  +0 =0, \; \forall \lambda_0>1.\]
    	 After the root $s=0$, the function $s \mapsto F(s,\lambda_0)$ increases reaching a maximum.  	 After that, it decreases until $\sqrt{\lambda_0} -1$, where it is negative, leaving a single root, which we denote $s^*=s^*(\lambda_0)$.  See Figure~\ref{fig:s-star-zero} for a typical illustration of  $F(s,7)$ for $\lambda_0=7$.
    	 \begin{figure}[H]
	\centering
	\includegraphics[width=0.3\linewidth]{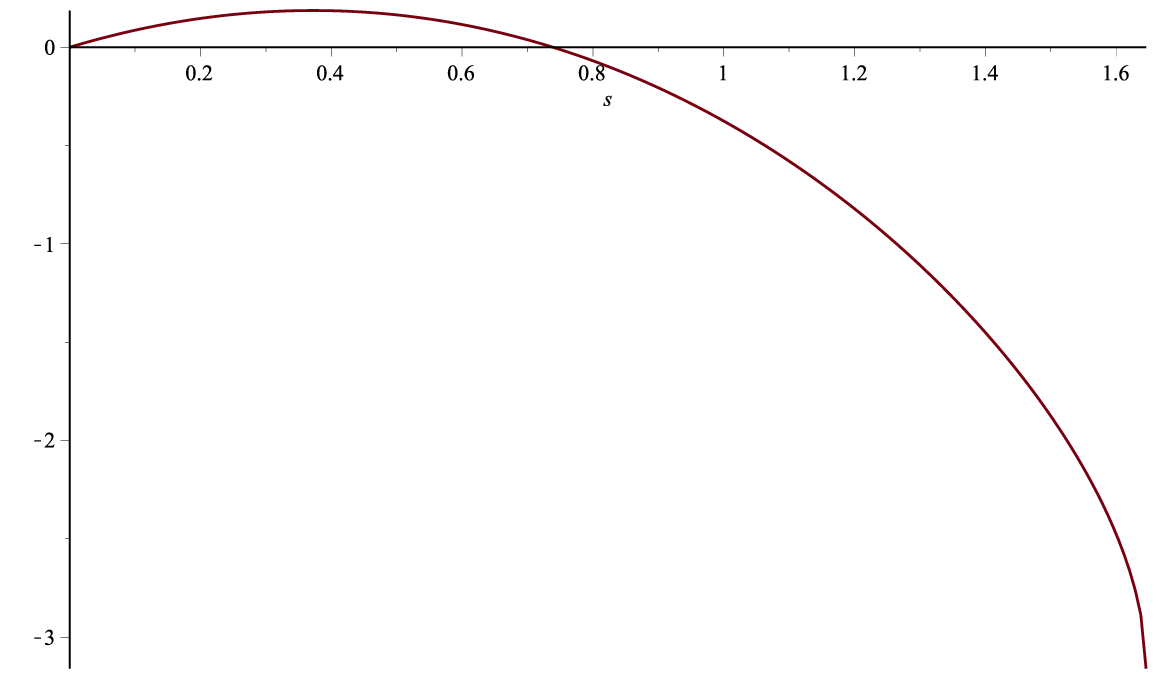}
	\caption{Graph of $F(s,7)$ where $s^*(7)\approx 0.7388$.}
	\label{fig:s-star-zero}
\end{figure}
 Knowing the existence and uniqueness of $s^*$ we can find a representation for it, by simplifying the equation $F(s^*, \lambda_0)=0$  and to solve it with respect to $s^*$. An extensive algebraic manipulation shows that   the pair $(s^*,\lambda_0)$ is the only positive solution of the following degree four polynomial equation
    	 {\small \[ -4\,\lambda\,{s}^{4}+ \left( 4\,{\lambda}^{2}-4 \right) {s}^{3}+
    	 	\left( -4\,{\lambda}^{3}+12\,\lambda-8 \right) {s}^{2}+ \left( 4\,{
    	 		\lambda}^{3}-12\,{\lambda}^{2}+12\,\lambda-4 \right) s=0.
    	 \]}
    	
    	 Using Cardano's formula we obtain four solutions: two complex numbers, zero (which is not suitable, as we already discuss) and  a positive real number  $s^*$ given by
    	 \[s^*(\lambda)=\left( \frac{\ell}{2}\,-{\frac {4\,{\lambda}^{2}+8\,\lambda-2}{\ell}}+{\lambda+1}\right)
    	 \frac{\lambda-1}{3 \lambda},
    	 \]
    	 where
    	 $\ell=\sqrt [3]{12\,\sqrt {3}\sqrt {{\frac{3\,{\lambda}^{3}+ 4\, {\lambda}^{2} +20\,\lambda-4}{\lambda}}}{\lambda}^{2}-28\,{\lambda}^{3}+24\, {\lambda}^{2}-48\,\lambda+8}$.
    	  \begin{figure}[H]
    	  	\centering
    	  	\includegraphics[width=0.4\linewidth]{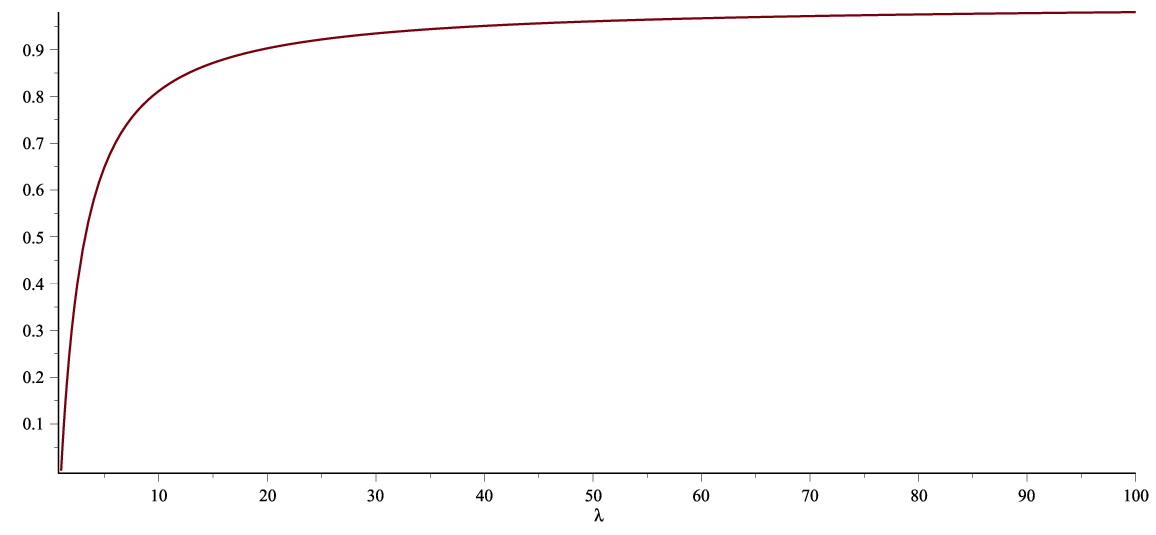}
    	  	\caption{Graph of $s^*(\lambda)$.}
    	  	\label{fig:s-star}
    	  \end{figure}	
    	  Hence, we can define the correspondence $\lambda \to s^*(\lambda)$ (see Figure~\ref{fig:s-star}).	The correspondence $\lambda \to s^*(\lambda)$ is increasing, $\lim_{\lambda \to 1} s^*(\lambda) = 0$ and $\lim_{\lambda \to \infty} s^*(\lambda) = 1$.
    	
    	  We can also inspect the behavior of $F(s, \lambda)$ in the whole domain, see Figure~\ref{fig:s-star-zero3d}
    	  \begin{figure}[H]
    	  	\centering
    	  	\includegraphics[width=0.3\linewidth]{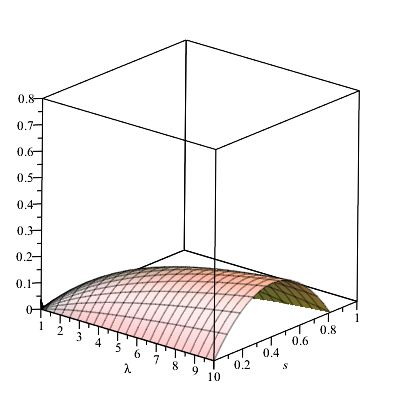}
    	  	\caption{Graph of $F(s,\lambda)$.}
    	  	\label{fig:s-star-zero3d}
    	  \end{figure}
    	  As $\lambda$ approaches $1$ we can see that the interval $(0, s^*)$ becomes narrow and vanishes at $\lambda=1$.
    \end{proof}

    We now show that the inequality obtained in Claim 1 implies the unboundness of the sequence $(\beta_j)_{j \in \mathbb{N}}$, given by Equation~\eqref{eq:beta}.\\
    \textbf{Claim 3:}  If $\frac{ s^2}{b_{j-1} ^2 }> 1$, for all $j$ then  the sequence $(\beta_j)_{j \in \mathbb{N}}$ is unbounded.
    \vspace{0.5cm}\\
    \begin{proof}    	
    	 We know that $(\beta_j)_{j \in \mathbb{N}}$ satisfy a variable coefficient first order  difference Equation~\eqref{eq:beta}, whose solution is given by Equation~\eqref{eq:explict beta}:
    	  \[\beta_{j} = c_{j} +  \gamma_{j} c_{j-1} + \gamma_{j}\gamma_{j-1} c_{j-2} + \cdots +  \gamma_{j}\gamma_{j-1}\cdots \gamma_{2}\beta_{1}.\]    	
    	  In order to evaluate $\beta_{j}$  we notice that
    	  $c_j:= \frac{1+ r_{j} s^2 \frac{ 1 }{(\lambda -1)^2} }{  -b_{j} }>\frac{1 }{  -b_{j}}> \frac{1 }{  \delta(s) -\theta'(s) } >0,\; j\geq 2$, and $\frac{1 }{  -b_{j}}> \frac{1 }{  \delta(s) -\theta'(s) } >0,\; j\geq 2$   implies that
    	  \begin{equation}\label{eq: lowr_bound}
    	  	 \frac{1}{b_{j} b_{j-m}}=  \left( \frac{1 }{  -b_{j}}\right) \left(  \frac{1 }{  -b_{j-m}}\right)  > \frac{1 }{  (\delta(s) -\theta'(s))^2 } >0.
    	  \end{equation}  	     	
    	
    	  Then, because $\gamma_j:=\frac{ s^2}{b_{j-1}b_{j}},\; j\geq 2 $, we can evaluate each product as follows
    $$\gamma_{j} \cdots \gamma_{j-m} = \frac{s^2}{b_{j} b_{j-m}} \; \frac{ s^2}{b_{j-1} ^2 } \cdots \frac{ s^2}{b_{j-m+1} ^2 }.$$
 From the hypothesis,  $\frac{ s^2}{b_{j-1} ^2 }> 1$ for all $j$, and from Equation~\eqref{eq: lowr_bound}, we derive  following inequality
\begin{equation}\label{eq:conv_cond}
\gamma_{j} \cdots \gamma_{j-m} = \frac{s^2}{b_{j} b_{j-m}} \; \frac{ s^2}{b_{j-1} ^2 } \cdots \frac{ s^2}{b_{j-m+1} ^2 }> \frac{s^2}{b_{j} b_{j-m}} \; >\frac{s^2 }{  (\delta(s) -\theta'(s))^2 },~~~m \geq 0.
\end{equation}
    	 Using Equations~\eqref{eq: lowr_bound} and \eqref{eq:conv_cond} we obtain
    	 \[\beta_{j} = c_{j} +  \gamma_{j} c_{j-1} + \gamma_{j}\gamma_{j-1} c_{j-2} + \cdots +  \gamma_{j}\gamma_{j-1}\cdots \gamma_{2}\beta_{1}>\]
    	 \[ > \frac{1 }{  \delta(s) -\theta'(s) } + (j-2) \frac{s^2}{  (\delta(s) -\theta'(s))^3}  + \frac{s^2 }{  (\delta(s) -\theta'(s))^2 }\beta_{1}.\]
    	  Hence,  that $\beta_j$ is unbounded because it tends to infinity when $j \to \infty$.
    \end{proof}

    \textbf{Claim 4:}   For any fixed $\lambda_0>1$ and  $0<s<s^*(\lambda_0)$ one has $\hat \lambda= \lim_{k \to \infty}  \rho(M_{T_{k}}(s))= \lambda$ for all $\lambda \geq \lambda_0$.
    \vspace{0.5cm}\\
    \begin{proof}   	
    	  From Claim 2, we know that for any fixed $\lambda_0>1$, there exists a unique value $0<s^*:=s^*(\lambda_0) <\sqrt{\lambda_0} -1$ such that  $F(s, \lambda_0)>0$ for each $0<s<s^*$.   From Claim 1,  if the function $F(s, \lambda_0)$ is positive then $\frac{ s^2}{b_{j-1} ^2 }> 1$, for all $j$. Finally, from Claim 3, if $\frac{ s^2}{b_{j-1} ^2 }> 1$, for all $j$ then  the sequence $(\beta_j)_{j \in \mathbb{N}}$, given by Equation~\eqref{eq:beta}, is unbounded. From the second part of Theorem~\ref{thm: main criterion} we know that,   if  the sequence $(\beta_j)_{j \in \mathbb{N}}$  is unbounded, then $\hat \lambda_0= \lim_{k \to \infty}  \rho(M_{T_{k}}(s))= \lambda_0$. In other words, $\lambda_0$ is an $s$-deformed Laplacian limit point.
    	
    	  Consider any $\lambda>\lambda_0$. Since  the correspondence  $\lambda \to s^*(\lambda)$ is increasing we have $s^*(\lambda) > s^*(\lambda_0)$, thus by the same above reasoning $\lambda$ is also an $s$-deformed Laplacian limit point, for the same $s$ as $\lambda_0$.     	
    	
    \end{proof}

Therefore, the proof of the theorem follows.

\end{proof}

\section{Remarks and numerical data}\label{sec: Remarks and numerical data}

\subsection{An approximating algorithm }
As an application, we can use $\operatorname{Diagonalize}(M_{T}(s), -\lambda)$ to compute an approximation of $\rho(M_{T}(s))$. This is useful because if $T=[r_1, r_2, \ldots,r_k]$ then $|T|=k+ \sum_{i=1}^{k}r_i$ reaches very large numbers in some cases (for instance, if $s\approx 0$).  This makes the traditional matrix computations non-feasible.

\begin{figure}[H]
	{\tt	{\footnotesize
		\begin{tabbing}
			aaa\=aaa\=aaa\=aaa\=aaa\=aaa\=aaa\=aaa\= \kill
		 \> Input: $T=[r_1, r_2, \ldots, r_k]$ a caterpillar\\
			\> Input: $A<B$ to be lower and upper bounds for $ \rho(M_{T}(s))$\\
			\> Input: $N \geq 2$ be the number of iterations \\
			\> Input: $s$ a real number\\
			\> Output: an approximation $\rho$ of $ \rho(M_{T}(s))$ \\
			\>  Algorithm Approximate $ \rho(M_{T}(s))$ \\
			\> \> {\bf for } for j from 1 to N do:\\
			\> \> \> {\bf  initialize the computation of } $\operatorname{Diag}(M_{T}(s), - \frac{A+B}{2})$\\
			\> \> \> {\bf if some output is non-negative then}  \\
			\> \> \>  \>  {\bf break the computation of} $\operatorname{Diag}(M_{T}(s), - \frac{A+B}{2})$\\
			\> \> \>  \> $A \leftarrow \frac{A+B}{2}$\\
			\> \> \>  \> $B  \leftarrow B$\\
			\> \> \>  \> $\rho:= \frac{A+B}{2}$\\
			\> \> \> {\bf else, if  if all the outputs of} $\operatorname{Diag}(M_{T}(s), - \frac{A+B}{2})$  {\bf are negative then}  \\
			\> \> \>  \> $A \leftarrow A$\\
			\> \> \>  \> $B  \leftarrow \frac{A+B}{2}$\\
			\> \> \>  \> $\rho:= \frac{A+B}{2}$\\
			\> \>  {\bf end loop}
		\end{tabbing}}}
	\caption{Computing the spectral radius of $T=[r_1, r_2, \ldots, r_k]$.}\label{fig:compute algorithm}
\end{figure}
The algorithm always work and is linearly fast with respect to the number of vertices of the backbone used in $\operatorname{Diagonalize}(M_{T}(s), - \frac{A+B}{2})$ (could be even faster in case of an early break). It performs successive bisections of the interval $[A,B]$ to seek the solution, thus $|\rho - \rho(M_{T}(s))| < \frac{A-B}{2^N}$. The bounds $A$ and $B$ are trivial for trees and could be easily obtained from the minimum and maximum degrees at $T$.  In the case of $T=T_{k}$ be in the Shearer's sequence we can choose $A=1$ and $B=\lambda$ because, by construction   $1< \rho(M_{T_{k}}(s)) < \lambda$. Also, if we  choose  $s \in (0,s^*(\lambda))$ (the natural choice for the Shearer's sequence), then we know that the result will indeed approximate $\lambda$.

In the examples, the speed of convergence (to $\lambda$) is defined by the growing rate of $\beta_j$ (recall that $\lambda-  \rho(M_{T_{k}}(s))=\varepsilon_k < \alpha_k=1/\beta_k$) which we know to be proportional to a sum of products  of the quantities $\frac{s^2}{b_j^2}>1$.

In particular, values close to $0$ and $s^*$ tend to be the least efficient  because $\frac{s^2}{b_j^2} \approx 1$ (see Figure~\ref{fig:s-star-zero}). Thus, a good choice to  speed improvement would be $s= \frac{s^*}{2}$. The best choice would be the local maximum of $F(\cdot,\lambda)$, but it depends on an additional (and not necessary) numerical approximation.

\bigskip

\subsection{Numerical data}
In all cases we employ the Algorithm Approximate $ \rho(M_{T}(s))$ from Figure~\ref{fig:compute algorithm}.

\begin{example} Consider $\lambda=1.5$ then $s^*=0.17869088+$ and choose $s= \frac{s^*}{2}= 0.08934544+$
	 \begin{center}
	 	\begin{tabular}{ccll}
	 		$k$ &	$T_{k}(\lambda)$  &	 $ \rho(M_{T_{k}}(s))$ & Error$=\lambda-  \rho(M_{T_{k}}(s)) $ \\
	 		\hline \\
	 		$5$ & $[20, 4, 0, 2, 9]$ & $1.499999827168959+$ & $1.72831041 \times 10^{-7}$ \\
	 		$10$ & $[20, 4, 0, 2, 9, 4, 1, 7, 11, 8]$ & $1.499999999999235+$ & $7.65 \times 10^{-13}$ \\
	 		$20$ &  $[20, 4, 0, \ldots,  6, 9, 6]$ & $1.499999999999999+$ & $2.68 \times 10^{-23}$ \\
	 		$30$ &  $[20,4,0, \ldots, 9, 10, 6]$ & $1.499999999999999+$ & $2.33 \times 10^{-34}$ \\
	 		$50$ &  $[20,4,0, \ldots, 3, 9, 3]$ & $1.499999999999999+$ & $7.26 \times 10^{-55}$ \\
	 	\end{tabular}
	 \end{center}
Now, choose $s= s^*$ to compare the speed of convergence
	\begin{center}
		\begin{tabular}{ccll}
			$k$ &	$T_{k}(\lambda)$  &	 $ \rho(M_{T_{k}}(s))$ & Error$=\lambda-  \rho(M_{T_{k}}(s)) $ \\
			\hline \\
			$5$ & $[4, 1, 0, 1, 2]$                                            & $1.49854070+$  & $1.459  \times 10^{-3}$ \\
			$10$ & $[4, 1, 0, 1, 1, 1, 2, 0, 1, 1]$                    &  $1.49986673+$ & $1.332 \times 10^{-4}$ \\
			$20$ &  $[4, 1, 0, 1, 1, 1, 2, 0, \ldots, 0, 1, 2]$  &  $1.49999959+$ & $4.035  \times 10^{-7}$  \\
			$50$ &  $[4, 1, 0, 1, 1, 1, 2, 0, \ldots, 0, 1, 2]$  &  $1.49999999+$ & $7.013  \times 10^{-17}$ \\
			$80$ &   $[4, 1, 0, 1, 1, 1, 2, 0, \ldots, 0, 1, 1]$ & $1.49999999+$  & $1.704  \times 10^{-26}$
		\end{tabular}
	\end{center}
	 We can see that  the number of pendant paths at the caterpillars are more sparse leading to a slower convergence.
\end{example}

\begin{example} Consider now $\lambda=5.4$ then $s^*=0.6718978+$ and choose $s= \frac{s^*}{2}= 0.3359489+$ to a fast convergence
	\begin{center}
		\begin{tabular}{ccll}
			$k$ &	$T_{k}(\lambda)$  &	 $ \rho(M_{T_{k}}(s))$ & Error$=\lambda-  \rho(M_{T_{k}}(s)) $ \\
			\hline \\
			$5$ & $[31, 23, 9, 17, 23]$                          & $5.39999978119+$ & $2.18  \times 10^{-7}$ \\
			$10$ & $[31, 23, 9,  \ldots, 12, 20, 22]$   & $5.39999999999+$ & $5.05  \times 10^{-14}$  \\
			$20$ &  $[31, 23, 9,  \ldots, 25, 19, 16]$  & $5.39999999999+$ & $4.10  \times 10^{-24}$  \\
			$50$ &  $[31,23,9, \ldots, 24, 24, 24]$     & $5.39999999999+$ & $2.18  \times 10^{-57}$ \\
			$80$ &   $[31,23,9, \ldots, 4, 22, 21]$      & $5.39999999999+$ & $2.43  \times 10^{-75}$
		\end{tabular}
	\end{center}
	We can see that the number of pendant paths at the caterpillars are more dense leading to a fast convergence.
	
	From \cite{lineartrees-24} we know that for the Laplacian matrix ($s=1$), the Shearer sequence does not approximate $\lambda=5.4$. We can see that by choosing $s$ close to 1.
		\begin{center}
		\begin{tabular}{ccll}
			$s$ & 	$T_{150}(\lambda)$  &	 $ \rho(M_{T_{150}}(s))$ & Error$=\lambda-  \rho(M_{T_{150}}(s)) $ \\
			\hline \\
			$0.9$ & $[4,1,2,  \ldots, 2,2,2]$          & $5.3999999999+$ & $4.99  \times 10^{-42}$ \\
			$0.99$ & $[3,1,1,  \ldots,1,2,1]$          & $5.3999999999+$ & $1.04  \times 10^{-29}$  \\
			$0.999$ &  $[3,1,1,  \ldots,1,1,2]$  & $5.3656282604+$ & $3.43  \times 10^{-2}$  \\
			$0.9999$ &  $[3,1,1,  \ldots,1,1,2]$  & $5.3716157858+$ & $2.83  \times 10^{-2}$
		\end{tabular}
	\end{center}
It is possible to prove that the best approximation will be $5.37+$ (see \cite[Lemma 3.5]{lineartrees-24}).
\end{example}
\begin{example} In this final example we unleash the full power of the combined algorithms we developed. Consider $\lambda=2025$ (because we are in 2025 and the convergence for big numbers is very fast). Computing $s^*:=0.9990125897+$ the Theorem~\ref{thm:suf criterion} says that for any $s \in (0, s^*)$ the Shearer sequence satisfy $\lim_{k \to \infty}  \rho(M_{T_{k}}(s))= \lambda$. As discussed before we took $s= \frac{s^*}{2}= 0.499506294892+$ to a fast convergence and $k=150$. The whole process of numerically approximate $s^*$,  generate the caterpillar
	\[T_{150}=  [8108, 7431, 8095, 8086, 8102, 8093, \ldots, 8095, 8090, 8102, 8102, 8092],\]
	approximate the spectral radius of the caterpillar $ \rho(M_{T_{150}}(0.499506294892+))$, via  Algorithm Approximate,  as being
	\[\rho(M_{T_{150}}(0.499506294892+))= 2024.9999999999999999999999\ldots999956+\]
	and the absolute error
	\[2025 -  \rho(M_{T_{150}}(0.499506294892+))= 4.332248  \times 10^{-193},\]
	took approximately $1.7$ seconds, using a naive implementation (a non-optimized code) in a personal computer. Notice that $|M_{T_{150}}(0.499506294892+)|= 1211693$ thus a traditional approach would be required to deal with a matrix with more than a million as dimension, a hart task, despite the fact that such a matrix being highly sparse.
\end{example}

\subsection{Remarks and open questions}
We still do not know what happens when $s^*(\lambda_0) $ approaches 1 (the classical Laplacian matrix). Our main result, Theorem~\ref{thm:suf criterion}, claims that the interval $[\lambda_0, \; +\infty)$ is entirely formed by $s$-deformed Laplacian limit points  for $s<s^*(\lambda_0)$, but in this situation $\lambda_0$ becomes arbitrarily big.

However, one could still have $\hat \lambda= \lim_{k \to \infty}  \rho(M_{T_{k}}(s))= \lambda$  for $s> s^*(\lambda_0)$. In this case, our proof for Theorem~\ref{thm:suf criterion}  shows that, for some terms $\frac{s^2}{b_j^2} <1$ so the terms $\beta_j$ could still be unbounded due some compensation of  big and small values of $\frac{s^2}{b_j^2} $ in Equation~\ref{eq:conv_cond}.

If we were able to extend  $ s^*(\lambda_0) \to 1$ without $\lambda_0 \to \infty$ we would be able to find an interval of Laplacian limit points ($s=1$) by taking the limit with respect to $s$. This question remains open.

\section*{Acknowledgements.}
Elismar Oliveira and Vilmar Trevisan are partially supported by  CNPq grant 408180/2023-4. Vilmar Trevisan also acknowledges the support of CNPq grant 308774/2025-6.

\end{document}